\documentclass[11pt,reqno,tbtags,a4paper]{amsart}
\usepackage{amssymb}
\usepackage{xpunctuate}
\usepackage{url}
\usepackage[square,numbers]{natbib}
\bibpunct[, ]{[}{]}{;}{n}{,}{,}

\title
{Quasi-random graphs, subgraph counts and graph limits, again}  

\date{18 June, 2026}

\dedicatory{Dedicated to the memory of Vera T. S\'os}

\author{Svante Janson}
\thanks{Supported by the Knut and Alice Wallenberg Foundation
and
the Swedish Research Council
}
\address{Department of Mathematics, Uppsala University, PO Box 480,
SE-751~06 Uppsala, Sweden}
\email{svante.janson@math.uu.se}
\newcommand\urladdrx[1]{{\urladdr{\def~{{\tiny$\sim$}}#1}}}
\urladdrx{http://www2.math.uu.se/~svantejs}

\subjclass[2020]{} 

\numberwithin{equation}{section}

\renewcommand\le{\leqslant}
\renewcommand\ge{\geqslant}

\allowdisplaybreaks

\theoremstyle{plain}
\newtheorem{theorem}{Theorem}[section]
\newtheorem{lemma}[theorem]{Lemma}
\newtheorem{prop}[theorem]{Proposition}
\newtheorem{corollary}[theorem]{Corollary}
\newtheorem{conj}[theorem]{Conjecture}

\theoremstyle{definition}

\newcommand\xqed[1]{%
    \leavevmode\unskip\penalty9999 \hbox{}\nobreak\hfill
    \quad\hbox{#1}}

\newtheorem{exampleqqq}[theorem]{Example}
\newenvironment{example}{\begin{exampleqqq}}
  {\xqed{$\triangle$}\end{exampleqqq}}

\newtheorem{remarkqqq}[theorem]{Remark}
\newenvironment{remark}{\begin{remarkqqq}}
  {\xqed{$\triangle$}\end{remarkqqq}}

\newtheorem{definition}[theorem]{Definition}
\newtheorem{assumptions}[theorem]{Assumptions}
\newtheorem{problem}[theorem]{Problem}

\theoremstyle{remark}

\newenvironment{acks}{\section*{Acknowledgement}}{}

\newcounter{dummy}
\makeatletter
\newcommand\myitem[1][]{\item[#1]\refstepcounter{dummy}\def\@currentlabel{#1}}
\makeatother

\newenvironment{romenumerate}[1][-10pt]{
\addtolength{\leftmargini}{#1}\begin{enumerate}
 \renewcommand{\labelenumi}{\textup{(\roman{enumi})}}%
 \renewcommand{\theenumi}{\labelenumi}%
 }{\end{enumerate}}

\newenvironment{PXenumerate}[1]{
\begin{enumerate}
 \renewcommand{\labelenumi}{\textup{(#1\arabic{enumi})}}%
 \renewcommand{\theenumi}{\labelenumi}%
 }{\end{enumerate}}

\newcounter{oldenumi}
{\setcounter{oldenumi}{\value{enumi}}
\begin{romenumerate} \setcounter{enumi}{\value{oldenumi}}}
{\end{romenumerate}}

\newcounter{thmenumerate}

\newcounter{xenumerate}   

\newcommand\pfitemx[1]{\par#1:}
\newcommand\pfitemref[1]{\pfitemx{\ref{#1}}}

\newcommand\pfcase[2]{\smallskip\noindent\emph{Case #1: #2} \noindent}

\newcommand{\refT}[1]{Theorem~\ref{#1}}

\newcommand{\refC}[1]{Corollary~\ref{#1}}
\newcommand{\refCs}[1]{Corollaries~\ref{#1}}
\newcommand{\refL}[1]{Lemma~\ref{#1}}

\newcommand{\refR}[1]{Remark~\ref{#1}}

\newcommand{\refS}[1]{Section~\ref{#1}}
\newcommand{\refSs}[1]{Sections~\ref{#1}}
\newcommand{\refSS}[1]{Section~\ref{#1}}
\newcommand{\refP}[1]{Proposition~\ref{#1}}

\newcommand{\refD}[1]{Definition~\ref{#1}}
\newcommand{\refDs}[1]{Definitions~\ref{#1}}
\newcommand{\refE}[1]{Example~\ref{#1}}
\newcommand{\refEs}[1]{Examples~\ref{#1}}

\newcommand{\refApp}[1]{Appendix~\ref{#1}}
\newcommand{\refAss}[1]{Assumptions~\ref{#1}}
\newcommand{\refConj}[1]{Conjecture~\ref{#1}}

\begingroup
  \divide\count255 by 60
  \multiply\count255 by -60
  \advance\count255 by \time
  \ifnum \count255 < 10 \xdef\klockan{\the\count1.0\the\count255}
  \else\xdef\klockan{\the\count1.\the\count255}\fi
\endgroup

\DeclareMathOperator*{\sumx}{\sum\nolimits^{*}}

\newcommand{\sumim}{\sum_{i=1}^m}

\newcommand{\sumir}{\sum_{i=1}^r}
\newcommand{\sumiri}{\sum_{i=1}^{r-1}}
\newcommand{\sumjr}{\sum_{j=1}^r}

\newcommand\set[1]{\ensuremath{\{#1\}}}
\newcommand\bigset[1]{\ensuremath{\bigl\{#1\bigr\}}}
\newcommand\Bigset[1]{\ensuremath{\Bigl\{#1\Bigr\}}}

\newcommand\xpar[1]{(#1)}
\newcommand\bigpar[1]{\bigl(#1\bigr)}
\newcommand\Bigpar[1]{\Bigl(#1\Bigr)}
\newcommand\biggpar[1]{\biggl(#1\biggr)}
\newcommand\lrpar[1]{\left(#1\right)}

\newcommand\cpar[1]{\{#1\}}

\def\rompar(#1){\textup(#1\textup)}    
\newcommand\xfrac[2]{#1/#2}

\newcommand\parfrac[2]{\lrpar{\frac{#1}{#2}}}

\newcommand\innprod[1]{\langle#1\rangle}

\def\xexp(#1){e^{#1}}

\newcommand\floor[1]{\lfloor#1\rfloor}

\newcommand\ntoo{\ensuremath{{n\to\infty}}}
\newcommand\Ntoo{\ensuremath{{N\to\infty}}}

\newcommand\bmin{\land}

\newcommand\norm[1]{\lVert#1\rVert}

\newcommand\punkt{\xperiod}    

\newcommand{\aex}{a.e\punkt}

\newcommand\ii{\mathrm{i}}

\newcommand{\tend}{\longrightarrow}

\newcommand\LLto{\overset{L^2}{\tend}}

\newcommand\bbR{\mathbb R}
\newcommand\bbC{\mathbb C}
\newcommand\bbN{\mathbb N}

\newcommand\bbZ{\mathbb Z}

\newcounter{CC}
\newcounter{cc}

\newcommand\Tr{\operatorname{Tr}}

\newcommand\ga{\alpha}
\newcommand\gb{\beta}
\newcommand\gd{\delta}

\newcommand\gf{\varphi}

\newcommand\gs{\sigma}

\newcommand\eps{\varepsilon}
\renewcommand\phi{\xxx}  

\newcommand\cD{\mathcal D}
\newcommand\cE{\mathcal E}

\newcommand\cK{\mathcal K}

\newcommand\cM{\mathcal M}

\newcommand\cP{\mathcal P}

\newcommand\cX{{\mathcal X}}

\newcommand\tG{\widetilde G}

\newcommand\tP{\widetilde P}

\newcommand\bU{\overline U}

\newcommand\indic[1]{\boldsymbol1\cpar{#1}}

\newcommand\etta{\boldsymbol1}

\newcommand\qw{^{-1}}

\newcommand\qq{^{1/2}}

\newcommand\intoi{\int_0^1}

\newcommand\oi{\ensuremath{[0,1]}}

\newcommand\oio{\ensuremath{[0,1)}}

\newcommand\dd{\,\mathrm{d}}

\newcommand\lhs{left-hand side}
\newcommand\rhs{right-hand side}

\newcommand\gnq{\ensuremath{(G_n)}}

\newcommand\fS{\mathfrak{S}}
\newcommand\tfS{\widetilde{\mathfrak{S}}}
\newcommand\GG{\mathfrak{G}}
\newcommand\GL{\mathsf{GL}}
\newcommand\tgs{\widetilde\gs}

\newcommand\chimle[1]{\chi_{m,\le#1}}
\newcommand\chix{\widehat\chi}
\newcommand\ZZq{\binom}
\newcommand\ZZ[2]{\ZZq{[#1]}{#2}}
\newcommand\tZZq{\tbinom}
\newcommand\tZZ[2]{\tZZq{[#1]}{#2}}
\newcommand{\stirling}[2]{\genfrac{ \{ }{ \} }{0pt}{}{#1}{#2}}

\newcommand\NN{^{(N)}}
\newcommand\Nx[1]{[N]^{#1}}
\newcommand\Nm{\Nx{m}}
\newcommand\Nk{\Nx{k}}
\newcommand\nm{[n]^{m}}
\newcommand\sym{\textsf{s}}

\newcommand\hf{\widehat f}
\newcommand\bi{\mathbf{i}}
\newcommand\bx{x}
\newcommand\by{y}
\newcommand\rhox[1]{[\rho(#1)]}
\newcommand\rhof{\rhox{f}}

\newcommand\comp{^{\mathsf{c}}}
\newcommand\normm[1]{\norm{#1}_2}
\newcommand\njgs{n_j(\gs)}
\newcommand\qchi{{\chi^\ast}}
\newcommand\qrho{{\rho^\ast}}

\newcommand\Rchixm{R_{\chix_m}}

\newcommand\DNm{D_{N,m}}
\newcommand\Dnm{D_{n,m}}
\newcommand\DcNx[1]{D\comp_{N,#1}}
\newcommand\DcNm{D\comp_{N,m}}
\newcommand\DcNk{D\comp_{N,k}}
\newcommand\cDNm{\cD_{N,m}}
\newcommand\cDcNm{\cD\comp_{N,m}}

\newcommand\ettDcNm{\etta_{\DcNm}}
\newcommand\ettDcNmi{\etta_{\DcNx{m-1}}}
\newcommand\ettDcNk{\etta_{\DcNk}}
\newcommand\ettcDNm{\etta_{\cDNm}}

\newcommand\ettcDcNm{\etta_{\cDcNm}}
\newcommand\pqr{$p$-\qr}

\newcommand\qr{quasi-random}

\newcommand\ff{{|F|}}
\newcommand\ffm{m}

\newcommand\prodik{\prod_{i=1}^k}
\newcommand\prodim{\prod_{i=1}^m}

\newcommand\prodir{\prod_{i=1}^r}

\newcommand\tN{\widetilde N}
\newcommand\uumxgs{U_{\gs(1)}\times\dots\times U_{\gs(m)}}
\newcommand\qrp{\qr{} property}

\newcommand\pqrp{\pqr{} property}

\newcommand\cPm{\cP_{\mmr}}
\newcommand\tcP{\widetilde\cP}
\newcommand\tcPm{\widetilde\cP_{\mmr}}

\newcommand\hh{\widehat h}

\newcommand\aam{A_1,\dots,A_m}

\newcommand\aamr{{A_1^{m_1}\times\dots\times A_r^{m_r}}}
\newcommand\AAr{A_1,\dots,A_r}

\newcommand\gaam{\ga_1,\dots,\ga_m}
\newcommand\gaar{\ga_1,\dots,\ga_r}

\newcommand\uum{U_1,\dots,U_m}
\newcommand\uumx{U_1\times\dots\times U_m}
\newcommand\uur{U_1,\dots,U_r}
\newcommand\mmr{{m_1,\dots,m_r}}
\newcommand\emmr{e^{(m_1,\dots,m_r)}}
\newcommand\uumr{U_1^{m_1}\times\dots\times U_r^{m_r}}

\newcommand\xxf{x_1,\dots,x_\ff}
\newcommand\xxm{x_1,\dots,x_m}

\newcommand\zzm{z_1,\dots,z_m}
\newcommand\zzr{z_1,\dots,z_r}

\newcommand\tPsi{\widetilde\Psi}

\newcommand\psifx[1]{\Psi_{F,#1}}
\newcommand\psifw{\psifx W}

\newcommand\tpsifx[1]{\widetilde\Psi_{F,#1}}
\newcommand\tpsifw{\tpsifx W}

\newcommand\mpx{measure-preserving}
\newcommand\mpp{\mpx}

\newcommand\sC{\mathsf{C}}
\newcommand\sK{\mathsf{K}}
\newcommand\tensor{\otimes}
\newcommand\refTK[1]{\refT{TK}\ref{TKii}\ref{#1}}
\newcommand\restr[1]{|{#1}}
\newcommand\PhiN{\Phi\NN}
\newcommand\fx{f_\ast}
\newcommand\Tnni{T_{N,n,\bi}\,}

\newcommand\Lm{L^{2}(\oi^m)}
\newcommand\Lsm{L^{2,\mathsf s}(\oi^m)}
\newcommand\Lmx[1]{L^{2}_{#1}(\oi^m)}
\newcommand\Lmk{\Lmx{k}}
\newcommand\Lmlex[1]{\Lmx{\le #1}}
\newcommand\Lmlek{\Lmlex{k}}
\newcommand\Lsmx[1]{L^{2,\mathsf{s}}_{#1}(\oi^m)}
\newcommand\Lsmk{\Lsmx{k}}
\newcommand\Lsmlex[1]{\Lsmx{\le #1}}
\newcommand\Lsmlek{\Lsmlex{k}}

\newcommand\Ls{L^{2,\sym}}
\newcommand\Lx[1]{L^{2}_{#1}}

\newcommand\Llex[1]{\Lx{\le #1}}

\newcommand\Lsx[1]{L^{2,\sym}_{#1}}
\newcommand\Lsk{\Lsx{k}}
\newcommand\Lslex[1]{\Lsx{\le #1}}
\newcommand\Lslek{\Lslex{k}}

\newcommand\lls{\ell^{2,\mathsf s}}
\newcommand\llslex[1]{\ell^{2,\mathsf s}_{\le#1}}
\newcommand\llslek{\llslex{k}}

\newcommand\tell{\widetilde{\ell}}

\newcommand\llll{\ell^2}

\newcommand\CS{Cauchy--Schwarz}
\newcommand\CSineq{\CS{} inequality}

\newcommand{\Lovasz}{Lov\'asz}

\begin{document}

\begin{abstract} 
We study properties of graphs (or rather graph sequences)
saying that some restricted count of subgraphs is approximatively
what is expected in a random graph.
It has been shown by several authors that many such properties 
characterize quasi-random graphs, but there are also some exceptions.
We continue here the line of investigation in Janson and S\'os (2013),
and introduce some new versions of these properties, in order to better
understand why many of these properties are quasi-random, and to understand
the structure of the exceptions that are not.
A new feature in the proofs is a simple decomposition of the subspace of
symmetric functions in $L^2([0,1]^m)$ into subspaces that are irreducible
for the action of measure-preserving transformations of $[0,1]$;
this simplifies some arguments and gives structure to others.
\end{abstract}

\maketitle

\section{Introduction}\label{S:intro}

My only paper with Vera T. S\'os \cite{SJ294} was about
quasi-random graphs and graph limits, and more precisely 
about some properties of subgraph counts that were shown to characterize 
quasi-random graphs. We left also a number of open problems and conjectures.
Some of these were answered by \citet{Hatami2Li}.
Here we consider the same type of properties again, and obtain further results
extending and complementing results in \cite{SJ294} and \cite{Hatami2Li}.
There are still cases left open and new open problems and conjectures;
we hope that the method and point of view in the present paper can contribute
to the understanding of this type of properties.

We assume that the reader has some familiarity with \qr{} graphs, but we
repeat some basic facts.
Recall that \emph{\qr} is an asymptotic property, so it really is a property
of sequences $(G_n)_n$ of graphs (with $|G_n|\to\infty$ as \ntoo).
\citet{Thomason87a,Thomason87b} and \citet{ChungGW:quasi} showed
that a number of different ``random-like'' properties of 
a graph sequence $(G_n)$
are equivalent, and we say that \gnq{} is \emph{\qr}, or more precisely
\emph{\pqr}, if 
it satisfies these properties. (Here $p\in\oi$ is a
parameter.)
We further say that a property of
sequences $(G_n)_n$ of graphs (with $|G_n|\to\infty$) is a 
\emph{\qrp} (or  a \emph{\pqrp}) if it 
characterizes \qr{} (or \pqr) sequences of graphs.
Many \qr{} properties of very different types
have  been added by various authors. 

When the theory of graph limits was introduced by
\citet{LSz} and \citet{BCLSV1,BCLSV2}, 
it became clear that 
a sequence $(G_n)_n$ of graphs is \pqr{} if and only if $G_n\to p$
in the sense of graph limits, where $p$ denotes the graphon that is constant
$p$ \cite{LSz}.
Both the present paper and \cite{SJ294} are based  on this relation.

In the present paper we consider \qr{} properties 
that are stated in terms of subgraph counts
(typically with some restrictions). 
Properties of this type correspond in great generality to corresponding
properties of graphons, such that the question whether a given property 
of graph sequences is \qr{} 
can be translated to the question whether the corresponding property of 
graphons characterizes constant graphons.
This has been a fundamental idea in \cite{SJ234},
\cite{SJ294}, and \cite{Hatami2Li}. 
In the present paper we study some further properties of subgraph counts
and the corresponding properties of graphons. 

It was noted in \cite[Remark 9.6]{SJ294} that
one ingredient in several proofs there was the study of some
subspaces of $L^2(\oi^m)$ 
that are invariant under all measure-preserving bijections of $\oi$ onto itself,
and in particular some such subspaces consisting of symmetric functions only.
It was conjectured that all such subspaces 
are direct sums of some set of certain explicit spaces.
The main purpose of the present paper is to prove this conjecture from
\cite{SJ294} (\refT{T1} below), and to combine it with tools from previous works
to obtain new results on \qr{} properties.

The proof of \refT{T1} uses group representation theory, and in particular some
representations of the finite symmetric groups on finite-dimensional subspaces
of $L^2(\oi)^m$, which intuitively
may be regarded as approximations of the
infinite-dimensional representation of the infinite group $\GG$
of \mpp{} bijections.
(See \refApp{Arep} for a summary of relevant representation theory.)

\refT{T1} helps us to analyze our main problem,
i.e., whether certain subgraph count properties are \qr, by allowing us to
consider these irreducible spaces of functions only.
In particular, this makes it possible to split the problem whether the
properties studied here are \qr{} into two parts that can be studied
separately, see \refR{Rsplit}.
Nevertheless, as seen in \cite{SJ294} and \cite{Hatami2Li},
while this in many cases leads to a simple proof that a certain property is
\qr, there are also cases where 
this only leads to a reformulation as a non-trivial algebraic question;
this question was solved in some cases in \cite{SJ294}, 
but remains open in other cases.
We do not solve this algebraic problem here; on the contrary, our analysis
below adds in some cases  new versions of it that are open.
(See \refS{Sbad}.)
Another purpose of the present paper is to draw attention to these problems,
and to stimulate research on them and on what they say about the 
general structure of the type of \qr{} properties studied here.

Some further comments and open problems are given in \refS{Sfurther}.

\section{Notation and some background}\label{Smain}

\subsection{General notation}\label{SSnot}
We let $p\in(0,1)$ be a fixed parameter, usually omitted from the notation.
We denote the number of elements of a finite set $F$ by $|F|$.

All graphs in this paper are finite, undirected and simple.
The vertex and edge sets of a graph $G$ are denoted by
$V(G)$ and $E(G)$. We write $|G|:=|V(G)|$ for the
number of vertices of $G$, and $e(G):=|E(G)|$ for the
number of edges.

We let 
$\bbN:=\set{1,2,\dots}$ and,
as usual, $[n]:=\set{1,\dots,n}$.
For any finite set $S$,
we let $\ZZq Sm$ denote the set of all $\binom {|S|}m$ subsets of $S$ of
size $m$. 

For a vector $x=(x_1,\dots,x_m)$
and a subset $I=\set{i_1,\dots,i_k}\subseteq[m]$,
we write
\begin{align}\label{da3}
  x_I:=\xpar{x_{i_1},\dots,x_{i_k}}.
\end{align}
 
We let
$x\bmin y:=\min(x,y)$ for $x,y\in\bbR$.

$\indic{\cE}$ denotes the indicator function of an event $\cE$, and if $A$
is a set, we also write $\etta_A(x):=\indic{x\in A}$.

We will consider 
(complex-valued)
functions on $\oi^k$ and subsets of $\oi^k$ for $k\ge0$;
all such functions and subsets are tacitly assumed to be (Lebesgue) measurable.
($\oi^0$ is a one-element set.)
As usual, functions that are equal a.e.\ are identified.
(We some times write ``a.e.'' for emphasis, but this is often omitted.)
Integrals are with respect to Lebesgue measure (in one or several
dimensions), and  
we usually omit ``$\dd x $''. 
The Lebegue measure of a set $A$ is denoted by $|A|$.
(There should be no risk of confusion with the notation for finite sets.)

For $f$ and $g$ in $L^2(\oi^m)$ (or in a subspace thereof), we denote their
inner product by $\innprod{f,g}:=\int f\overline{g}$, and the $L^2$-norm by
$\norm{f}_2:=\innprod{f,f}\qq$.
We let $f_n\LLto f$ denote convergence in $L^2$.

$\fS_N$ is the symmetric group of all $N!$ permutations of \set{1,\dots,N}. 

$\GG$ is the (infinite) group of all measure-preserving bijections
$\gf:\oi\to\oi$.

\subsection{Notation for functional analysis and representation theory}
\label{SSnotF}

All our Hilbert spaces will be complex.
Let $H$ be a Hilbert space with inner product $\innprod{f,g}$.
Then $f\perp g$ means that $f$ and $g$ are orthogonal, i.e., $\innprod{f,g}=0$.

If $L$ is a (closed) subspace of a Hilbert space $H$, then
\begin{align}\label{xi1}
L^\perp:=\set{f\in H:f\perp g \;\forall g\in L}.  
\end{align}
If furthermore $M\subseteq L$ is another subspace, then
\begin{align}\label{xi2}
L\ominus M := L\cap M^\perp = \set{f\in L:f\perp g\;\forall g\in\cM}.  
\end{align}
Two subspaces $L,M\subseteq H$ are orthogonal, denoted by $L\perp M$,
if $f\perp g$ for all $f\in L$ and $g\in\cM$.
In this case, we denote their sum by $L\oplus M:=\set{f+g:f\in L,\, g\in\cM}$.

A unitary operator $T$ on $H$ is a bijective linear map $T:H\to H$  such
that $\norm{Tf}=\norm{f}$ for every $f\in H$, which is equivalent to
$\innprod{Tf,Tg}=\innprod{f,g}$ for all $f,g\in H$.

 A unitary representation of a group $G$ on $H$ is a group
homomorphism $g\mapsto\rho_g$ where each $\rho_g$ is a unitary operator on
$H$,
see \refApp{Arep}.

Let $\rho$ be a unitary representation of $G$ on $H$.
A closed subspace $W\subseteq H$ is \emph{invariant} if $\rho_g(W)\subseteq W$
for every $g\in G$.
Then the restriction of $\rho_g$ to $W$ defines a representation
$G$ on $W$, which we denote by $\rho\restr{W}$. (Sometimes we simply use
$\rho$ also for the retriction.)

Given a representation $\rho$ of $G$ on $H$, and an element $f\in H$,
we define
\begin{align}\label{hull}
  \rhof := \text{the closed linear hull of }\set{\rho_g(f):g\in G}
\subseteq H.
\end{align}
This is an invariant subspace of $H$; in fact, $\rhof$
is the smallest closed invariant subspace that contains $f$.

A representation $\rho$ on $H$  is \emph{irreducible} if the only closed
invariant 
subspaces are the trivial $\set{0}$ and $H$.

\subsection{Subgraph counts}\label{SSsub}

Throughout the paper, $F$ is a fixed labelled graph, and $m:=|F|$.
We assume (without loss of generality) that
$V(F)=[m]=\set{1,\dots,m}$. 
For another labelled graph $G$, we define the following subgraph counts.

\begin{definition}\label{DN}
Let $\uum$ be subsets of $V(G)$, where (as said above) $m=|F|$.
  \begin{romenumerate}
\item\label{Nuu}
$N(F,G;\uumx)$ is the number of 
labelled copies of $F$ in $G$ with the $i$th vertex
in $U_i$; equivalently, $N(F,G;\uumx)$ is the number of
injective graph homomorphisms 
$\gf:F\to G$ such that $\gf(i)\in U_i$ for every 
$i\in V(F)=[m]$. 

\item \label{tNuu}

$\tN(F,G;\uumx)$ is the symmetrized version obtained
by taking the average over all labellings of $F$; equivalently,
\begin{equation}\label{tN}
  \tN(F,G;\uumx):=\frac{1}{\ffm!}\sum_{\gs\in\fS_m}N(F,G;\uumxgs),
\end{equation}
summing over all permutations $\gs$ of $\set{1,\dots,m}$.

  \end{romenumerate}
\end{definition}
Note that for the 
version in \ref{Nuu},
the labelling of $F$ and the ordering of $\uum$ are important,
but for the
symmetrized version in \ref{tNuu} these do not matter,
and therefore \eqref{tN} is defined also for
unlabelled $F$, although we for conveníence will regard $F$ as labelled
there too.
Note also that if $F$ is the complete graph $\sK_m$, we have
$\tN(\sK_m,\dots)=N(\sK_m,\dots)$.

In the present paper, we will mainly study  the
subgraph count
\begin{align}
  \tN(F,G;\uumr)
\end{align}
where we are given $r$ disjoint subsets $U_1,\dots,U_r$ and 
the subset $U_i$ is repeated $m_i$ times; here $1\le r\le m=|F|$ and
$m_1,\dots,m_r\in\bbN$ with $\sumjr m_j=m$ are given.
In this case, 
$\tN(F,G;U_1^{m_1},\dots,U_r^{m_r})$
equals, 
up to the unimportant symmetry factor $\prod_i m_i!/m!$,
the number of unlabelled
copies of $F$ in $G$ that have exactly $m_i$ vertices in $U_i$.
(For each such copy of $F$, there are $\prod_i m_i!$ labellings of
$F$ for which it is counted, and the total number of labellings of $F$ is
$m!$.) 

We make the following standing assumptions (which we may repeat for clarity).
\begin{assumptions}\label{AAA}
$F$ is a (labelled) graph with $e(F)>0$;\\\phantom{\qquad}
$m=\ff$;\\\phantom{\qquad}
$1\le r\le m$;\\\phantom{\qquad}
$(\mmr)$ is a vector of positive integers with $\sumjr m_j=m$;
\\\phantom{\qquad}
$(\gaar)$ is a vector of positive real
numbers with $\sumim\ga_i\le1$;\\ \phantom{\qquad}
  $p\in(0,1)$.
\end{assumptions}
Note that, to avoid some trivial exceptions later, we have excluded the cases
$e(F)=0$, $p=0$, and $p=1$. 
The quantities in \refAss{AAA} are regarded as fixed,
and may be omitted from the notation. 
(In particular, we usually omit $p$, since all results are the same for
every $p\in(0,1)$.)

\begin{definition}\label{DP}
Assume \refAss{AAA}.
We define the following properties
of a graph sequence $\gnq$.
  \begin{romenumerate}[-10pt]
  \item \label{DP1}
 $\cPm(F;\gaar)$ is the property
that
\begin{equation}\label{dp1}
  N\bigpar{F,G_n;\uumr}=p^{e(F)}\prodir|U_i|^{m_i}+o\bigpar{|G_n|^\ffm}.
\end{equation}
holds for all disjoint
subsets $\uur$ of $V(G_n)$ with 
\begin{align}\label{dp0}
|U_i|=\floor{\ga_i|G_n|}, 
\qquad i=1,\dots,r.  
\end{align}

\item 
$\tcPm(F;\gaar)$ is the property that 
\begin{equation}\label{dp2}
  \tN\bigpar{F,G_n;\uumr}=p^{e(F)}\prodir|U_i|^{m_i}+o\bigpar{|G_n|^\ffm}.
\end{equation}
holds for all  $\uur$ as in \ref{DP1}.
  \end{romenumerate}
\end{definition}

Here $U_1,\dots,U_r$ depend on $n$, although we omit this from the notation,
and the error terms $o(|G_n|^m)$ in \eqref{dp1} and \eqref{dp2} are supposed
to hold uniformly for all permitted choices of $U_i=U_{i,n}$.

\begin{remark}\label{RP=P?}
  The relation between the properties of the types $\cP$ and $\tcP$ is not
  clear, except the simple observation that they coincide 
in the cases $r=1$ or $F=K_m$; furthermore
$\cPm(F;\gaar)$ implies $\tcPm(F;\gaar)$
if $m_1=\dots=m_r$ and $\ga_1=\dots=\ga_r$.
See the comments in \cite[Remarks 2.9 and 2.14]{SJ294}.

In the present paper, we mainly study $\tcP$.
\end{remark}

The case $r=1$ was studied  by 
\citet{SS:nni}; note that in this case, 
$\tN(F,G;U^m)=N(F,G;U^m)$ by \eqref{tN}, and hence 
$\tcP_m(F;\ga) = \cP_m(F;\ga)$.
More precisely, \cite{SS:nni} did not restrict the size of the subset,
and showed that \eqref{dp1} (or \eqref{dp2}) for all subsets $U\subseteq V(G)$
is a \qrp.
\citet{Shapira} and \citet{Yuster} 
assumed further that $|U|=\floor{\ga|G_n|}$ for some fixed $\ga$ with
$0<\ga<1$; they showed (\cite{Shapira} for $\ga=1/(\ffm+1)$ and \cite{Yuster}
in general) that, in our notation in \refD{DP},
$\cP_m(F;\ga)$ is a \qrp.
(The case $F=\sK_2$ and $\ga=1/2$ is already in \citet{ChungGW:quasi}.)  

Another case studied by several authors is $r=m$, when $m_1=\dots=m_r=1$ and
we assume that $U_1,\dots,U_m$ are disjoint.
We note first a simple but important example and counterexample.

\begin{example}\label{ERK2}
Taking $F=\sK_2$, $r=m=2$, and $\ga_1+\ga_2=1$, 
the property $\tcP_{1,1}(\sK_2,\ga,1-\ga)=\cP_{1,1}(\sK_2,\ga,1-\ga)$ 
is obviously equivalent to the property of cuts 
that the number of edges in $G_n$ between a subset $U\subseteq V(G_n)$ 
and $\bU:=V(G)\setminus U$
is $\bigpar{p+o(1)}|U|\,|\bU|$ whenever $|U|=\floor{\ga |G_n|}$.
%
It was noted by \citet{ChungGW:quasi} that this is \emph{not} 
a \qrp{} in the symmetric case $\ga=\xfrac12$.
On the other hand,
\citet{ChungG} showed that this is a \qrp{}
for every $\ga\neq 1/2$,
see also \cite[Section~9]{SJ234}.

For the case $\ga_1+\ga_2<1$,
\citet{SS:nni} showed that $\cP_{1,1}(\sK_2,1/3,1/3)$ is a \qrp.
More generally, $\cP_{1,1}(\sK_2,\ga_1,\ga_2)$ for any $\ga_1,\ga_2>0$ with
$\ga_1+\ga_2<1$ is a \qrp{} as a special case of \cite[Theorem 2.11]{SJ294}.
\end{example}

Further results for the case $r=m$
have been given
by \citet{Shapira}, \citet{ShapiraYuster,ShapiraYuster:hyper},
\citet{Yuster}; in particular
it follows from \cite[Lemma 2.2]{Shapira} and \cite{Yuster}
that $\tcP_{1,\dots,1}(F;\ga,\dots,\ga)$ is a \qrp{} for any $\ga<1/m$,
and 
it is shown in 
\cite{ShapiraYuster:hyper} 
that $\cP_{1,\dots,1}(K_m;\ga_1,\dots,\ga_m)$
is a \qrp{} for every  
$m\ge2$ and $(\ga_1,\dots,\ga_m)\neq(1/m,\dots,1/m)$ with $\sumim \ga_i=1$.
More generally, \citet[Theorem 2.11]{SJ294} showed that
$\tcP_{1,\dots,1}(F;\ga_1,\dots,\ga_m)$ is a \qrp{}
for any $F$ with $e(F)>0$
and 
$(\ga_1,\dots,\ga_m)\neq(1/m,\dots,1/m)$ with $\sumim \ga_i\le1$.
\citet{Hatami2Li} gave a different proof of this result.


The case $(\ga_1,\dots,\ga_m)=(1/m,\dots,1/m)$ remains somewhat mysterious.
(The results below might  give more understanding of \emph{why} this case is
exceptional.) 
The methods in \cite{SJ294} and \cite{Hatami2Li} reduce this case to an
algebraic problem for graphs, which seems surprisingly difficult, see 
\refL{Lalg}.
As we have noted in \refE{ERK2}, 
$\cP_{1,\dots,1}(F;1/m,\dots,1/m)$ is \emph{not} a \qrp{} for $F=\sK_2$;
furthermore, it is shown by \citet{ShapiraYuster:hyper} 
that a related hypergraph cut property fails for $(\gaam)=(1/m,\dots,1/m)$.
On the other hand, 
\citet{HuangLee} show that $\cP_{1,\dots,1}(F;1/m,\dots,1/m)$ is a \qrp{} 
for $F=\sK_m$ for any $m\ge3$. 
It was conjectured in \cite{SJ294} that 
$\tcP_{1,\dots,1}(F;1/m,\dots,1/m)$ is a \qrp{} whenever $e(F)>1$,
and this was proved for the special cases when $F$ is regular, a star, or
disconnected \cite[Theorem 2.12]{SJ294}.

The case $1<r<m$ was considered briefly in \cite[Section 9]{SJ294},
where in particular the case $\tcP_{2,1}(F;\ga_1,\ga_2)$ was studied as an
example.
It was shown that this is a \qrp{} when $\ga_1+\ga_2<1$, but the case
$\ga_1+\ga_2$ was left open. 
More generally, the case $r=2\le m$ was considered in \cite{Hatami2Li}, 
where a conjecture from \cite{SJ294} was shown, implying in particular 
(as said in \cite{SJ294}) that 
$\tcP_{2,1}(F;\ga_1,\ga_2)$ is a \qrp{} also when $\ga_1+\ga_2=1$, provided
$\ga_1\notin\set{\frac13,\frac23}$. 

We return to these results, and their relation with the present approach, in
\refSs{Sgood} and \ref{Sbad}.

\section{Transfer to graphons}\label{Stransfer}

We assume that the reader is familiar with the theory of graphons,
see e.g.\ \cite{BCLSV1,BCLSV2,Lovasz}.
We use the standard version of graphons, defined as symmetric functions
$W:\oi^2\to\oi$. 

If $F$ is a labelled graph and $W$ a graphon, we define
\begin{equation}\label{psifw}
  \psifw(\xxf):=\prod_{ij\in E(F)} W(x_i,x_j).
\end{equation}
and its symmetrization
\begin{equation}
  \label{symm}
\tpsifw(\xxm):=\frac1{\ff!}\sum_{\gs\in\fS_{\ff}}
\psifw\bigpar{x_{\gs(1)},\dots,x_{\gs(\ff)}}
.\end{equation}

We can now define a graphon analogue of \refD{DP}.
(We deliberately use the same notation, justified by \refL{Lequiv} below.)

\begin{definition}\label{DW}
Assume \refAss{AAA}.
We define the following properties
of a graphon $W$.
  \begin{romenumerate}[-10pt]
  \item \label{DW1}
$\cPm(F;\gaar)$ is the property
that
\begin{equation}\label{dw1}
\int_{\aamr}   \psifw(\xxf)
=p^{e(F)}\prodir|A_i|^{m_i}
\end{equation}
holds 
for all disjoint subsets $\AAr$ of $\oi$ with 
\begin{align}\label{dw0}
|A_i|=\ga_i, 
\qquad i=1,\dots,r.  
\end{align}

\item 
$\tcPm(F;\gaar)$ is the property that 
\begin{equation}\label{dw2}
  \int_\aamr \tpsifw(\xxm)=p^{e(F)}\prodir|A_i|^{m_i}
\end{equation}
for all  $\aam$ as in \ref{DW1}.
  \end{romenumerate}
\end{definition}

\begin{definition}
  \label{D3}
A property of graphons $W$ is \emph{\qr} if every graphon $W$ that satisfies
it is \aex{} equal to a constant.
Furthermore, the property is \emph{\pqr} if it is satisfied only by graphons
$W$ that are \aex{} equal to $p$.
\end{definition}

We can now use standard arguments to translate our problem from graph
sequences to graphons. 

\begin{lemma}\label{Lequiv}
Assume \refAss{AAA}.
Then
the  property 
$\cPm(F;\gaar)$ of graph sequences
is \pqr{} if and only if
the property 
$\cPm(F;\gaar)$ of graphons is.

Similarly, 
the  property 
$\tcPm(F;\gaar)$ of graph sequences
is \pqr{} if and only if
the property 
$\tcPm(F;\gaar)$ of graphons is.
\end{lemma}

\begin{proof}
  This follows by standard arguments as in the special case $r=m$ 
in \cite[Lemma 3.4]{SJ294}, using a straightforward extension of
\cite[Lemma 3.1]{SJ294} and \cite[Lemma 7.2]{SJ234}.
\end{proof}

\section{Representations of the group of \mpp{} bijections}

In the present paper, we attack the problem of whether the properties in
\refDs{DP} and \ref{DW} are \qr{} by finding certain irreducible subspaces of
$L^2(\oi^m)$ 
for the natural representation of 
the group of \mpp{} bijections of $\oi$ onto itself.
The idea behind this is that graphons are defined up to \mpp{} bijections, 
so properties of them are invariant under such bijections; hence, we are
interested in invariant subspaces, so it seems natural to study such
subspaces in general.
In the present section we state the 
result on representations
that we will use;
the proof is given
in \refS{Spf}.

We are really interested in the functions $\psifw$ 
and $\tpsifw$ on $\oi^m$ defined in \eqref{psifw}, 
which are bounded and real-valued (in fact, with values in $\oi$). 
However, to have access to the general theory of group representations in a
(complex) Hilbert space, we will more generally consider functions in
$L^2(\oi^m)$, and we allow the functions to be complex-valued.

Our main result in this section (\refT{T1}) is for symmetric functions only.
Therefore  in the sequel we will mainly consider only the symmetric function
$\tpsifw$ and the corresponding property $\tcPm(F;\gaar)$.

\subsection{Some spaces of functions}\label{SS1}

We study functions in 
$L^2(\oi^m)$, the Hilbert space of complex-valued
square integrable functions on $\oi^m$, for some integer $m\ge0$.
(For $m=0$, this is by definition the one-dimensional space $\bbC$ of
constants.) 
We use the following notation for this space and certain subspaces,
where $m\ge0$ and $-1\le k\le m$.
We may sometimes omit the argument $\oi^m$.
Recall the notation $x_I$ from \eqref{da3}.
\begin{itemize}

\item 
$\Lsm:=$ the subspace of symmetric functions in $\Lm$, i.e. functions $f\in\Lm$
such that $f(x_1,\dots,x_m)=f(x_{\gs(1)},\dots,x_{\gs(m)})$ for every
$\gs\in\fS_m$. 

\item 
$\Lmlek:=$ the closed subspace of $\Lm$ spanned by functions $f(x)$
that depend only on a
subset $x_I$ of the coordinates for some $I\subseteq[m]$ with $|I|\le k$.
(Recall the notation \eqref{da3}.)
For $k=-1$ this means that $\Lmlex{-1}=\set0$.
Note also that $\Lmlex{m}=\Lm$.

\item 
$\Lsmlek:=\Lsm\cap \Lmlek$, the subspace of  symmetric functions in $\Lmlek$.
\item 
$\Lsmk:=\Lsmlek\ominus\Lsmlex{k-1}$
(for $0\le k\le m$ only).
\end{itemize}
%
Note that we have
\begin{align}
\set0&=\Llex{-1}\subseteq \Llex{0}\subseteq\dots\subseteq \Llex{m-1}
\subseteq \Llex{m}=\Lm,\label{d2a}\\
\set0&=\Lslex{-1}\subseteq \Lslex{0}\subseteq\dots\subseteq \Lslex{m-1}
\subseteq \Lslex{m}=\Lsm\label{d2s}
.\end{align}
Hence, $\Lsmk$, $k=0,\dots,m$, are pairwise orthogonal subspaces of
$\Lsm$, and
\begin{align}\label{d3}
  \Lsmlex{\ell}:=\bigoplus_{k=0}^\ell \Lsmk.
\qquad 0\le \ell\le m.
\end{align}
Note also that
\begin{align}
\label{d2}
\Lmx{0}
=\Lmlex0
=\Lsmlex0=\Lsmx0=\bbC,  
\end{align}
the constant functions on $\oi^m$.

We can characterize $\Lmlek$ and $\Lmk$ as follows using 
Fourier transforms.
(This is nothing special for the Fourier transform; the lemma holds for 
expansions using any orthogonal basis $(\chi_i)_i$ of $L^2(\oi)$ with
$\chi_0=1$, but we have chosen to be concrete.)
We denote the Fourier coefficients of $f$ by
$\hf(n_1,\dots,n_m)$, where $(n_1,\dots,n_m)\in\bbZ^m$.

\begin{lemma}\label{LF}
\ \begin{romenumerate}
  \item \label{LFle}
$\Lmlek$ equals the space of functions $f\in \Lm$ such that 
\begin{align}\label{lfle}
\hf(n_1,\dots,n_m)=0 
\quad\text{unless }|\set{i\in[m]:n_i\neq0}|\le k.
\end{align}

  \item \label{LFles}
$\Lsmlek$ equals the space of functions $f\in \Lsm$ such that 
\eqref{lfle} holds.

  \item \label{LF=}
$\Lsmk$ equals the space of functions $f\in \Lsm$ such that 
\begin{align}\label{lf=}
\hf(n_1,\dots,n_m)=0 
\quad\text{unless }|\set{i\in[m]:n_i\neq0}|=k.
\end{align}
  \end{romenumerate}
\end{lemma}
\begin{proof}
\pfitemref{LFle}
  If $f(x)=f_I(x_I)$ for some $I\subseteq[m]$ with $|I|\le k$ and some
  $f_I\in L^2(\oi^k)$, then $\hf(n_1,\dots,n_m)=0$ whenever $n_i\neq0$ for
  some $i\notin I$; hence  \eqref{lfle} holds.
By the definition, such functions $f_I(x_I)$ span $\Lsmlek$, and thus
\eqref{lfle} holds for all $f\in\Lsmlek$.

Conversely, if \eqref{lfle} holds, then the Fourier expansion shows that
$f$ is a limit of linear combinations of functions 
$\prod_{j\in J}e^{2\pi n_j\ii x_j}$ for sets $J\subseteq[m]$ with $|J|\le k$,
and these products belong to $\Lmlek$ by the definition.

\pfitemref{LFles}
By \ref{LFle} and the definition of $\Lsmlek$.

\pfitemref{LF=}
A consequence of \ref{LFles} and the definition of $\Lsmk$, 
using Parseval's relation (i.e., the orthogonality of the Fourier series).
\end{proof}

For $0\le k\le m$, 
define a linear map 
$\Phi_{m,k}:L^2(\oi^k)\to \Lm$ by,
for $x_1,\dots,x_m\in\oi$,
\begin{align}\label{Phi}
(\Phi_{m,k}f)(x_1,\dots,x_m):=\sum_{I\in \ZZ mk} f(x_I),
\end{align}
where we 
sum over the $\binom mk$ subsets of $[m]$ of size $k$.
It is easy to see that 
\begin{align}
  \label{jun2}
\text{If $f\in\Ls(\oi^k)$, then $\Phi_{m,k}f\in\Lsm$}.
\end{align}

\begin{lemma}\label{LPhi}
  Let $0\le k\le m$.
Then $\Phi_{m,k}$ is an isomorphism 
of $\Ls(\oi^k)$ onto $\Lsmlek$,
and of the subspace
$\Lsk(\oi^k)$ onto $\Lsmk$.
\end{lemma}

\begin{proof}
Let $f\in\Ls_k(\oi^k)$  and consider the Fourier coefficient
\begin{align}\label{jun1}
  \widehat{\Phi_{m,k}f}(n_1,\dots,n_m)
=\sum_{I\in\ZZ mk}\int_{\oi^m} e^{-2\pi\ii\sum_j n_jx_j}f(x_I)
\dd x_1\dotsm\dd x_m.  
\end{align}
The integral in \eqref{jun1} vanishes if $n_j\neq0$ for some $j\notin I$
(by integrating first over $x_j$).
Hence we may assume that $J:=\set{j:n_j\neq0}\subseteq I$.
On the other hand, if $J\subsetneq I$, then 
$e^{2\pi\ii\sum_j n_jx_j}$ depends only on at most $k-1$ variables $x_j$,
and thus
$e^{2\pi\ii\sum_j n_jx_j}\in \Lmlex{k-1}$
is orthogonal to $f$ and again the integral is 0.
Hence, the Fourier coefficient in \eqref{jun1} vanishes unless $J=I$ for some
$I\in\ZZ mk$, i.e., if $|J|=k$, which shows that $\Phi_{m,k}f\in\Ls(\oi^k)$
by \refL{LF}\ref{LF=} and \eqref{jun2}.
Furthermore, in this case, for any $n_1,\dots,n_k\neq0$, we thus have
\begin{align}\label{jun3}
  \widehat{\Phi_{m,k}f}(n_1,\dots,n_k,0,\dots,0) & 
=\int_{\oi^m} e^{-2\pi\ii\sum_j n_jx_j}f(x_1,\dots,x_k)\dd x_1\dotsm\dd x_m
\notag\\&
=\hf(n_1,\dots,n_k).
\end{align}
The Fourier coefficient 
$\widehat{\Phi_{m,k}F}(n_1,\dots,n_m)$ is a symmetric function of
$n_1,\dots,n_m$ for any $F\in\Lsm$, and thus it follows
from \eqref{jun2}, \eqref{jun3}, and \refL{LF}\ref{LF=} that 
$\Phi_{m,k}$ is an isomorphism $\Lsk(\oi^k)\to\Lsmk$, and in fact an
isometry up to a constant factor $\binom mk\qq$ (the exact value is not
important for us).

Finally, this implies that
$\Phi_{m,k}$ also is an isomorphism $\Lslek(\oi^k)\to\Lsmlek$
by the decomposition \eqref{d3} applied to both sides (with $\ell=k$).
\end{proof}

For two functions $h_1$ and $h_2$ on $\oi$, we define
\begin{align}
  \label{hh}
h_1\tensor h_2(x,y):=h_1(x)h_2(y),
\qquad (x,y)\in\oi^2,
\end{align}
and similarly for more than two factors.
We write also $h^{\tensor k}:=h\tensor\dotsm\tensor h$ with $k\ge1$ factors.

Let $L^2_1(\oi):=\Ls_1(\oi)$, and note that,
by the definition or by \refL{LF}, 
\begin{align}\label{jun5}
L^2_1(\oi)=\Bigset{f\in L^2(\oi):\intoi f=0}.  
\end{align}

\begin{lemma}\label{L3}
Let $1\le k\le m$. Then
$\Lsmk$ is the closed linear span of
the set $\set{\Phi_{m,k}(h^{\tensor k}):h\in L^2_1(\oi)}$.
\end{lemma}

\begin{proof}
Let $L\subseteq \Ls(\oi^k)$.
be the closed linear  span of $\set{(h^{\tensor k}:h\in L^2_1(\oi)}$. 
  By \refL{LPhi}, it suffices to show that $L=\Lsk(\oi^k)$.

First, we have, for $h\in L^2_1(\oi)$,
  \begin{align}
    \widehat{h^{\tensor k}}(n_1,\dots,n_k)
= \prodik\hh(n_i),
  \end{align}
which vanishes if some $n_i=0$ since $\hh(0)=\intoi h=0$.
Hence, $h^{\tensor k}\in \Lsk(\oi^k)$ by \refL{LF}, and thus
$L\subseteq \Lsk(\oi^k)$.
For the converse, we note that if $h_1,\dots,h_k\in L^2_1(\oi)$,
then the symmetric function 
$\sum_{\gs\in\fS_k} h_{\gs(1)}\tensor\dots\tensor h_{\gs(k)}\in L$,
since by standard polarization it can be written as a linear combination of
some tensor powers, e.g.\ as 
(see e.g.\ \cite[(22)]{MazurO} and \cite[Appendix]{Thomas})
\begin{align}
  \sum_{\gs\in\fS_k} h_{\gs(1)}\tensor\dots\tensor h_{\gs(k)}
=\sum_{J\subseteq[k]}(-1)^{k-|J|}\Bigpar{\sum_{j\in J}h_j}^{\tensor k}
.\end{align}
Taking $h_i(x)=e^{2\pi\ii n_ix}$ with $n_i=0$ and using \refL{L1}\ref{LF=} 
shows that the Fourier expansion of any $f\in \Lsk(\oi^k)$ can be written as
a norm convergent sum of functions in $L$; hence
$\Lsk(\oi^k)\subseteq L$.
\end{proof}

\subsection{The decomposition into irreducible representations}
\label{SS2}

Let $\GG$ be the group of all measure-preserving bijections $\gf:\oi\to\oi$.
There is a natural unitary representation $\rho$ of $\GG$ on $\Lm$ given by
\begin{align}\label{rho}
  \rho_\gf(f)(x_1,\dots,x_m) = f(\gf\qw(x_1),\dots,\gf\qw(x_m)).
\end{align}
(This representation depends on the parameter $m$, but $m$ will be clear
from the context and we usually omit it from the notation.)

It is obvious that the subspaces of $\Lm$ defined in \refSS{SS1} 
all are invariant
subspaces of $\Lm$ for the representation $\rho$, and we can thus consider
the representation of $\GG$  on any of them given by the restriction of
$\rho$.
The basis of our results in the present paper is the following 
description of the (closed) invariant subspaces of $\Lsm$.

\begin{theorem}\label{T1}
  The representation \eqref{rho} of $\GG$ on $\Lsmk$ is irreducible for 
every $0\le k\le m$, and these $m+1$ representations are non-equivalent.
Furthermore, every closed subspace $M$  
of $\Lsm$ 
that  is invariant under the representation $\rho$
is of the form
\begin{align}\label{t1}
M=  \bigoplus_{k\in K} \Lsmk
\end{align}
for some set $K\subseteq\set{0,\dots,m}$.
In particular, if there exists a non-constant function
$f\in M$, then there
exists $k\in\set{1,\dots,m}$ such that $\Lsmk\subseteq M$.
\end{theorem}
This theorem might be known in other contexts, 
but we have been unable to find a reference
so we give a complete (and lengthy)  proof  in 
\refS{Spf}.

\section{Applications to subgraph count properties}\label{Sgood}

Assume \refAss{AAA}.

Let $\cM=\cM(\mmr;\gaar)$
be the subset of $\Lsm$ consisting of all functions $f\in\Lsm$ such that
\begin{equation}\label{m2}
  \int_\aamr f(\xxm)=0
\end{equation}
for all disjoint subsets $\AAr$ of $\oi$ with $|A_i|=\ga_i$, $i=1,\dots,r$.
Then $\cM$ is evidently a closed subspace of $\Lsm$ that is invariant under
the action of $\GG$, and
the property $\tP_\mmr(F;\gaar)$
is precisely, by comparing  \eqref{dw2} and \eqref{m2},
\begin{align}\label{mac1}
  \tpsifw-p^{e(F)} \in\cM.
\end{align}
Hence we can try to analyse the property
$\tcPm(F;\gaar)$ in two steps; we first find $\cM$, and then try to show that
if \eqref{mac1} holds, then necessarily $W=p$ a.e.

We use \refT{T1} and the orthogonal decomposition \eqref{t1} of $\cM$; let
 $\cK\subseteq\set{0,\dots,m}$ be the set $K$ in \eqref{t1} for this subspace
 $\cM$.  
We say that an integer $k$ is \emph{bad} if $k\in \cK$.
Thus
\begin{align}\label{cm}
\cM=  \bigoplus_{k\in \cK} \Lsmk,
\end{align}
taking the direct sum over all bad $k$.

\begin{remark}\label{Rsplit}
Note that $\cM$, $\cK$, and ``bad''  depend on $\mmr$ and $\gaar$, but not
on the graph $F$.
On the other hand, the property \eqref{mac1}, and whether it implies
$W=p$ and thus is \qr, depends on $F$ and $\cK$ (by \eqref{cm}), but not on
$r$, $\mmr$, or $\gaar$ except through $\cK$.
We have thus split the problem into two, connected only through the set $\cK$.
\end{remark}

It is shown in \cite[Proof of Theorem 2.11]{SJ294} that
\begin{align}\label{mac2}
    \tpsifw-p^{e(F)} =0\text{ a.e.\ on } \oi^m
\implies
W=p \text{ a.e.\ on } \oi^2.
\end{align}
This yields a simple sufficient criterion:

\begin{theorem}\label{T0}
Suppose that 
$\cK=\emptyset$, i.e., there is no bad $k$.
(This is by \eqref{cm} equivalent to 
the invariant subspace $\cM(\mmr;\gaar)$ defined by \eqref{m2}
being equal to \set0.)
Then the   property $\tcPm(F;\gaar)$ is \qr.
\end{theorem}
\begin{proof}
  In this case, \eqref{mac1} implies $W=p$ by \eqref{mac2}.
\end{proof}

In the sequel, 
\refT{T0} should be remembered as soon as we find 
$\cK=\emptyset$ (and thus $\cM=\set0$); we will
sometimes explicitly repeat the conclusion of \refT{T0}, but sometimes this
is left to the reader.
As we will see, the case $\cK=\emptyset$ is
generic, so \refT{T0} covers many cases; nevertheless there are, as noted in
\cite{SJ294} and \cite{Hatami2Li}, exceptional cases, which makes it
interesting to study these exceptional cases further.

More generally, we have the following.
\begin{theorem}\label{TKX}
  Let $\cK$ be the set of bad $k$, for the given $\mmr$ and $\gaar$.
Then $\tcP_\mmr(F;\gaar)$ is a \qrp{} if and only if
\begin{align}\label{tkx}
    \tpsifw-p^{e(F)} \in 
\bigoplus_{k\in \cK} \Lsmk
&\implies
W=p \text{ a.e.\ on } \oi^2,
\intertext{or, equivalently,}
\label{tky}
    \tpsifw-p^{e(F)} \in 
\bigoplus_{k\in \cK} \Lsmk
&\implies
W \text{ is constant a.e.\ on $\oi^2$}.
\end{align}
\end{theorem}
\begin{proof}
  As said above,
the property $\tP_\mmr(F;\gaar)$ is equivalent to \eqref{mac1}, 
and thus \eqref{cm} yields  the result  with \eqref{tkx}.
Finally, if 
$\tpsifw-p^{e(F)} \in \bigoplus_{k\in \cK} \Lsmk$, then by \refL{L0} below,
$\tpsifw-p^{e(F)}\perp\Ls_0(\oi^m)=\bbC$ and thus 
\begin{align}\label{bak}
\int_{\oi^m}\bigpar{\tpsifw-p^{e(F)}}=0.   
\end{align}
If $W=c$ a.e.\ for some constant $c$,
then $\tpsifw=c^{e(F)} $ a.e.\, and it follows from \eqref{bak} that $c=p$. 
Hence, \eqref{tky} is equivalent to \eqref{tkx}.
\end{proof}

In the remainder of this section we study the set $\cK$, with the hope of
showing that $\cK=\emptyset$ (so that \refT{T0} applies)
when possible, and otherwise to restrict $\cK$
as much as possible.
We return to consequences of \refT{TKX} when $\cK$ is non-empty in the next
section.

\begin{lemma}\label{L0}
  $0$ is never bad.
Hence $\cK\subseteq[m]$.
\end{lemma}
\begin{proof}
  If 0 were bad, 
then we would have $\bbC=\Lsx0\subseteq \cM$, but \eqref{m2} does not hold
  for $f=1$.
\end{proof}

We can convert the analytic condition \eqref{m2} into an algebraic condition
by introducing the elementary symmetric polynomials
\begin{align}\label{ele1}
  e_k(\zzm):=\sum_{1\le i_1<\dots<i_k\le m} z_{i_1}\dotsm z_{i_k}.
\end{align}
We define further
\begin{align}\label{ele2}
  \emmr_k(\zzr):=e_k(z_1,\dots,z_1,z_2,\dots,z_2,\dots,z_r,\dots,z_r)
\end{align}
with $z_i$ repeated $m_i$ times. 
Explicitly, by elementary combinatorics,
\begin{align}\label{ele33}
  \emmr_k(\zzr):=
\sum_{\ell_1+\dots+\ell_r=k} \prodir \binom{m_j}{\ell_j}z_j^{\ell_j}.
\end{align}

We then have the following result.
(We explicitly state several closely related versions, partly for later
convenience.) 
\begin{theorem}\label{TK}
  \begin{romenumerate}
    
  \item\label{TKi} 
If\/ $\sumir\ga_i<1$, then $\cK=\emptyset$ and no $k$ is bad.

  \item\label{TKii}
 If\/ $\sumir\ga_i=1$, then 
  the following are equivalent, for every integer $k\in[m]$,

  \begin{romenumerate}
 \renewcommand{\labelenumii}{\textup{(\alph{enumii})}}%
 \renewcommand{\theenumii}{\labelenumii}%
 \renewcommand{\theenumi}{}%

  \item\label{TK1}     
$k$ is bad.
\item \label{TK2}
$\Lsmk\subseteq \cM$.
\item \label{TK3}
Every $f\in\Lsmk$ satisfies \eqref{m2}
for all disjoint subsets $\AAr$ of $\oi$ with $|A_i|=\ga_i$, $i=1,\dots,r$.

\item\label{TK4} 
$\emmr_k(\zzr)=0$ 
for any $(\zzr)\in\bbC^r$ with $\sumir \ga_iz_i=0$.

\item\label{TK6}
The homogeneous polynomial, in $r-1$ variables $z_1,\dots,z_{r-1}$,
\begin{align}\label{ste1}
  \emmr_k\Bigpar{z_1,z_2,\dots,z_{r-1},-\frac{1}{\ga_r}\sum_{i=1}^{r-1}\ga_iz_i}=0
\end{align}
identically.

\item \label{TK5}
The homogeneous polynomial $\emmr_k(\zzr)$ 
of degree $m$ in $r$ variables is divisible
by $\sumir \ga_i z_i$
in the ring $\bbR[z_1,\dots,z_r]$ of polynomials. 
  \end{romenumerate}
\end{romenumerate}
\end{theorem}

\begin{proof}
We begin by considering both cases \ref{TKi} and \ref{TKii} together, thus
letting $\sumir\ga_i\le1$ until further notice.

\ref{TK1}$\iff$\ref{TK2}:
From \eqref{t1} and the definition of bad.

\ref{TK2}$\iff$\ref{TK3}:
By the definition  of $\cM$.

\ref{TK2}$\iff$\ref{TK4}:
Recall that $\cM$ is a closed subspace. Hence,
by \refL{L3}, \ref{TK2} is equivalent to \eqref{m2} holding for
$f=\Phi_{m,k}(h^{\tensor k})$ for every $h\in L^2_1(\oi)$.
We have, by \eqref{Phi}
\begin{align}\label{jun6}
  \Phi_{m,k}(h^{\tensor k})(x_1,\dots,x_m)
=\sum_{I\in\ZZ mk}h^{\tensor k}(x_I)
=\sum_{I\in\ZZ mk}\prod_{i\in I} h(x_i)
.\end{align}
Hence,
for any subsets $\aam\subseteq \oi$ with positive measures $\ga_i=|A_i|$,
\begin{align}\label{jun61}
  \int_{A_1\times\dots\times A_m}  \Phi_{m,k}(h^{\tensor k})
&= \sum_{I\in\ZZ mk}\int_{A_1\times\dots\times A_m}\prod_{i\in I} h(x_i)
\dd x_1\dotsm\dd x_m
\notag\\&
= \sum_{I\in\ZZ mk}\prod_{i\in I} \int_{A_i} h(x_i)\dd x_i
\cdot \prod_{i\notin I} |A_i|
\notag\\&
= \prodim \ga_i\cdot
\sum_{I\in\ZZ mk}\prod_{i\in I} \frac{1}{\ga_i}\int_{A_i} h(x_i)\dd x_i
\notag\\&
= \prodim \ga_i\cdot e_k(z_1,\dots,z_m)
,\end{align}
where we define
\begin{align}\label{zh}
  z_i:= \frac{1}{\ga_i}\int_{A_i} h(x) \dd x.
\end{align}
For subsets $A_1,\dots,A_r$ with each $A_i$ repeated $m_i$ times, we thus
obtain, 
\begin{align}
 \int_{A_1^{m_1}\times\dots\times A_r^{m_r}} \Phi_{m,k}(h^{\tensor k})
= \prodir \ga_i^{m_i}\cdot \emmr_k(z_1,\dots,z_r)
,\end{align}
Hence, condition \ref{TK2} is equivalent to
\begin{align}\label{jun7}
\emmr_k(z_1,\dots,z_r)=0,   
\end{align}
with $z_i$ given by \eqref{zh}, for every $h\in
L^2_1(\oi)$ and any disjoint subsets $\AAr$ of $\oi$ with $|A_i|=\ga_i$.

We now consider \ref{TKi} and \ref{TKii} separately.
In \ref{TKi}, i.e., the case  $\sumir\ga_i<1$,
we choose disjoint sets $A_i\subset\oi$ 
with $|A_i|=\ga_i$.
Then, for any $\zzr\in\bbC$, we may define $h$ by $h(x):=z_i$ for $x\in A_i$,
and let  
$h(x):=c$ for $x\in\oi\setminus\bigcup_{i=1}^r A_i$, where $c$ is chosen
such that $\intoi h=0$ and thus $h\in L^2_1(\oi)$.
Hence, if \ref{TK2} holds, then \eqref{jun7} holds for any
$\zzr\in\bbC$.
In particular, $\emmr_k(1,\dots,1)=0$, which contradicts the definition
\eqref{ele1}--\eqref{ele2}. 
This contradiction shows that \ref{TK2} cannot hold, which shows \ref{TKi}
by the equivalence \ref{TK1}$\iff$\ref{TK2} proved above.

In the remainder of the proof we consider case \ref{TKii}, and assume thus
$\sumir\ga_i=1$. We still choose  disjoint sets $A_i\subset\oi$ 
with $|A_i|=\ga_i$, but now $\bigcup_{i=1}^r A_i=\oi$ (up to a null set).
Hence, \eqref{zh} implies, for $h\in L^2_1(\oi)$,
\begin{align}\label{jun8}
0=  \intoi h = \sumir\int_{A_i} h = \sumir \ga_i z_i,
\end{align}
Conversely, if $\zzr\in\bbC$ with $\sumir\ga_i z_i=0$, then we may again
define $h$ by $h(x)=z_i$ for $x\in A_i$, and then $h\in L^2_1(\oi)$.

Consequently, the argument above shows that \ref{TK2} is equivalent to
\eqref{jun7} for all $\zzr$ with $\sumir\ga_iz_i=0$, which is \ref{TK4}.

\ref{TK4}$\iff$\ref{TK6}:
The condition $\sumir \ga_iz_i=0$ is equivalent to
\begin{align}\label{ste2}
  z_r=-\frac{1}{\ga_r}\sumiri \ga_iz_i,
\end{align}
and we obtain \eqref{ste1} by substituting \eqref{ste2} into $\emmr_k(\zzr)=0$.

\ref{TK4}$\implies$\ref{TK5}:
Regard $\emmr_k(\zzr)$ as a polynomial in $z_r$ with coefficients in the
ring $\bbR[z_1,\dots,z_{r-1}]$. By division with the linear polynomial
$\sumir\ga_iz_i=\ga_r\bigpar{z_r+\sum_{i=1}^{r-1}(\ga_i/\ga_r)z_i}$,
we see that there exist polynomials $Q(z_1,\dots,z_{r-1},z_r)$ and
$R(z_1,\dots,z_{r-1})$, in $r$ and $r-1$ variables, respectively, such that
\begin{align}\label{jun18}
\emmr_k(\zzr)=
  Q(z_1,\dots,z_{r})\sumir\ga_iz_i + R(z_1,\dots,z_{r-1}).
\end{align}
If \ref{TK4} holds, then \eqref{jun18} implies that
$R(z_1,\dots,z_{r-1})=0$ when $\sumir\ga_iz_i=0$, which means that
$R(z_1,\dots,z_{r-1})=0$ for any $z_1,\dots,z_{r-1}=0$; hence, $R=0$
as a polynomial, and \eqref{jun18} shows \ref{TK5}.

\ref{TK5}$\implies$\ref{TK4}:
Obvious.
\end{proof}

\begin{corollary}\label{CKi}
If\/ $\sumir\ga_i<1$, 
then the  property $\tcPm(F;\gaar)$ is \qr.
\end{corollary}
\begin{proof}
  An immediate consequence of Theorems \ref{TK}\ref{TKi} and \ref{T0}.
\end{proof}

In the sequel, we often consider  only the case 
$\sumir\ga_i=1$, since otherwise \refT{TK}\ref{TKi} and \refC{CKi} apply.

We consider a number of special cases, beginning with the cases $r=1$ and
$r=m$ treated earlier by various authors.
The case $r=1$ is quite special, and only the case $\ga_1=1$ is not covered
by \refT{TK}\ref{TKi} and \refC{CKi}.

\begin{corollary}\label{Cr=1}
  If $r=1$, then the following holds.
  \begin{romenumerate}
  \item\label{Cr=1i} 
If\/ $\ga_1<1$, then $\cK=\emptyset$ and no $k$ is bad.
  \item\label{Cr=1ii} 
If\/ $\ga_1=1$, then $\cK=[m]$ and every $k>0$ is bad.
  \end{romenumerate}
\end{corollary}

\begin{proof}
  If $\ga_1<1$, this follows by \refT{TK}\ref{TKi}.

The result for $\ga_1=1$ follows also from \refT{TK}, 
for example using \refTK{TK5} since $e^{(m)}_k(z)=\binom mk z^k$;
it can also easily be seen directly from \eqref{m2} which reduces to
$\int_{\oi^m}f=0$, i.e., $f\perp 1$.
\end{proof}

\begin{remark}\label{Rr=1}
  \refC{Cr=1} shows by \refT{T0} that $\cP_m(F;\ga_1)=\tcP_m(F;\ga_1)$ 
is a \qrp{} when $0<\ga_1<1$,
which as said above was proved by 
\citet{Shapira} and \citet{Yuster}. 
The case $\ga_1=1$ means that in \refD{DP} we consider only $U_1=V(G_n)$ and
thus count all subgraphs of $G_n$ isomorphic to $F$. 
It is well-known that the resulting property $\cP_m(F;1)$  is
not a \qrp{} for any fixed $F$; see \cite{ChungGW:quasi} and \cite{SS:nni} for
counterexamples. 
To see this
is easy using the graphon version
in \refD{DW}. It suffices to note that \eqref{dw1} reduces to the single
equation
\begin{align}\label{ele1p}
  \int_{\oi^m}\psifw = p^{e(F)}.
\end{align}
Every graphon $W$ satisfies this for some $p\in\oi$, and conversely, given
$p\in(0,1)$, we can easily find a non-constant graphon $W$ such that
\eqref{ele1p} holds.

On the other hand, the conjunction of $\cP_{|F|}(F;1)$
for certain sets of graphs $F$ may be a \qrp; one well-known example, 
shown already in \cite{ChungGW:quasi}, 
is $\set{\sC_4,\sK_2}$;
furthermore, $\cP_{|F|}(F;1)$ for \emph{all} graphs $F$ is the standard
definition of $G_n\to p$ 
in graph limit sense \cite{LSz,BCLSV1,Lovasz}.
\end{remark}

Also the case $r=m$ is simple.
\begin{corollary}\label{Cr=m}
Suppose $m=r\ge2$ (so $m_i=1\,\forall i$). 
Then no $k\ge2$ is bad, and we have
\begin{align}\label{cr=m}
  \cK=
  \begin{cases}
    \set1, 
\text{ and thus }\cM=\Ls_1(\oi^m),
& \text{if\/ }(\gaam)=(\frac1m,\dots,\frac1m);\\
\emptyset, \text{ and thus }\cM=\set0,& \text{otherwise}.
  \end{cases}
\end{align}
\end{corollary}

Combined with \refT{T0}, this gives a new proof 
that $\tcP_{1,\dots,1}(F;\gaam)$ is a \qrp{}
when $(\gaam)\neq(\frac1m,\dots,\frac1m)$,
which is part of
\cite[Theorem 2.11]{SJ294}.

\begin{proof}
First, assume $2\le k\le r=m$.
  Let $z_i:=1$ for $1\le i<k$, $z_k=-\sum_{i=1}^{k-1}\ga_i/\ga_k$, 
and $z_i:=0$ for $k<i\le r$. Then $\sumir \ga_i z_i=0$ and
\begin{align}
  e^{1,\dots,1}_k(\zzr)
= e_k(\zzr) = z_1\dotsm z_k\neq0.
\end{align}
Hence \refTK{TK4} does not hold, and thus $k$ is not bad.

For $k=1$ we have $e^{(1,\dots,1)}_1(\zzr)=e_1(\zzr)=z_1+\dots+z_r$.
We may assume $\sumir\ga_i=1$ by \refT{TK}\ref{TKi},
and then it follows from \refTK{TK5} 
that 1 is bad if and only if
$(\gaam)=(\frac1m,\dots,\frac1m)$, see also \refC{C1} below.
\end{proof}

The result for $k=1$ in \refC{Cr=m} is easily generalized to any $r$ and $m$.

\begin{corollary}\label{C1}
$1$ is bad if and only if
\begin{align}\label{c1}
  \ga_i=\frac{m_i}{m},
\qquad i\in[r].
\end{align}
\end{corollary}
\begin{proof}
We may assume $\sumir\ga_i=1$ by \refT{TK}\ref{TKi}.
  For $k=1$, the definitions \eqref{ele1}--\eqref{ele2} yield
  \begin{align}\label{c1aq}
    \emmr_1(\zzr)=\sumir m_i z_i.
  \end{align}
Hence \refT{TK}\ref{TKii}\ref{TK5} holds if and only if $m_i=c\ga_i$ for
some $c\in\bbC$. Since $\sumir\ga_i=1$ and $\sumir m_i=m$, we must have
$c=m$, which yields \eqref{c1}.
\end{proof}

For $r=2$, we obtain the following explicit result, shown (in an equivalent
formulation) by \citet{Hatami2Li}.

\begin{corollary}[{\cite[Theorem 3.7]{Hatami2Li}}]\label{Cr=2}
  Let $r=2$ and $\ga_1+\ga_2=1$. Then $k\in[m]$ is bad if and only if
  \begin{align}\label{ele3}
\sum_{\ell=0}^k\binom{m_1}{k-\ell}\binom{m_2}{\ell}
\parfrac{-\ga_1}{1-\ga_1}^\ell
=0.
  \end{align}
\end{corollary}

\begin{proof}
In the case $r=2$, \eqref{ele33} yields
\begin{align}\label{ele4}
  e_k^{(m_1,m_2)}(z_1,z_2)
=\sum_{\ell=0}^k\binom{m_1}{k-\ell}\binom{m_2}{\ell}
z_1^{k-\ell}z_2^{\ell}
.\end{align}
By homogeneity, \refTK{TK6} holds if and only if
$  e^{(m_1,m_2)}_k(1,-\ga_1/\ga_2)=0$, which by \eqref{ele4}
is equivalent to \eqref{ele3}.
\end{proof}

\begin{example}\label{Er=2a}
\citet[Corollary 3.8]{Hatami2Li} noted that in the special case $m_2=1$, 
it follows from \eqref{ele3} that $k\in[m-1]$ is bad if and only if 
$\ga_1=1-k/m$ (and thus $\ga_2=k/m$).
\end{example}

\begin{example}\label{Er=2}
  Let $m=4$, $r=2$, and $m_1=m_2=2$, and assume $\ga_1+\ga_2=1$. 
\refC{C1} (or \eqref{ele3}) shows that $1$ is bad if and only if
$\ga_1=\frac12$. 
For $k=2$, \eqref{ele3} yields, with $\gb=\ga_1/\ga_2$,
\begin{align}
    \gb^2-4\gb+1=0
\end{align}
with the roots $\gb=2\pm\sqrt3$, and it follows that 2 is bad if and only if
$\ga_1=(3\pm\sqrt 3)/6$.
For $k=3$, \eqref{ele3} yields $2\gb^2-2\gb=0$ which yields $\gb=1$ and 
thus $\ga_1=\frac12$. 
For $k=4$, \eqref{ele3} never holds, so 4 is not bad.
(See also \refCs{Cm} and \ref{C6} below.)
Note that if $\ga_1=\frac12$, then both 1 and 3 are bad.

Consequently, if $\ga_1\notin\set{\frac12,\frac12\pm\frac{\sqrt3}6}$,
then $\cM=\set0$ and thus $\tcP_{2,2}(F;\ga_1,1-\ga_1)$ is a \qrp{} for every 
graph $F$ with $|F|=4$ and $e(F)>0$.
\end{example}

For larger $r$ we have some only partial results.
We begin by showing that the case $\cM=\set0$
is generic in the sense that there are only isolated exceptions.

\begin{corollary}\label{Cfinite}
  For any given $m$, $r$, and $m_1,\dots,m_r$, there is at most a finite
  number of vectors $(\ga_1,\dots,\ga_r)\in(0,1)^r$ such that
  $\cK\neq\emptyset$. For any other $(\gaar)$ we thus have $\cM=\set0$ and 
then the property $\tcPm(F;\gaar)$ is \qr.
\end{corollary}
\begin{proof}
By \refT{TK}\ref{TKi} we may assume $\sumir\ga_i=1$.

The case $r=1$ is thus trivial, with $\ga_1=1$.

If $r=2$, then \refC{Cr=2} shows that for each $k\in[m]$,
if $k$ is bad then there is a non-trivial
polynomial equation for $-\ga_1/(1-\ga_1)$; hence there is only a finite
number of possible $\ga_1$, 
and thus a finite number of $(\ga_1,\ga_2)=(\ga_1,1-\ga_1)$.

If $r\ge3$, consider only $\zzr$ with $z_2=\dots=z_r$. 
Then $\emmr_k(z_1,\dots,z_r) =  e^{(m_1,m')}_k(z_1,z_2)$ with $m':=m-m_1$,
and the condition $\sumir \ga_i z_i=0$ becomes $\ga_1z_1+(1-\ga_1)z_2=0$.
Hence, assuming $k$ is bad, we obtain as in \refC{Cr=2} the polynomial
equation \eqref{ele3}, with $m_2$ replaced by $m'=m-m_1$, and thus only a
finite number of possible $\ga_1$. The same holds by symmetry for every $\ga_i$.

The final sentence follows by \eqref{cm} and \refT{T0}.
\end{proof}

\begin{corollary}\label{Cm}
  If\/ $r\ge2$, 
then $m$ is not bad.
\end{corollary}
\begin{proof}
  We have $e_m(\zzm)=z_1\dotsm z_m$, and thus
  \begin{align}
    \emmr_m(\zzr)=\prodir z_i^{m_i}.
  \end{align}
We may choose $z_i\neq0$ with $\sumir \ga_i z_i=0$, and thus 
\refTK{TK4} does not hold.
\end{proof}

More generally we have the following bound.

\begin{remark}\label{Rperm}
$\cM(\mmr;\gaar)=\cM(m_{\gs(1)},\dots,m_{\gs(r)};\ga_{\gs(1)},\dots,\ga_{\gs(r)})$
for any permutation $\gs\in\fS_r$.
Hence, we may by simultaneous permutations of $\mmr$ and $\gaar$ without
loss of generality assume, for example, $m_1\ge\dots\ge m_r$.
The results below do not assume this, but the formulations
are sometimes for convenience optimized for this case.
\end{remark}

\begin{corollary}\label{C6}
  Suppose $r\ge2$. If $k$ is bad, then $k<m_{r-1}+m_r$.
\end{corollary}
\begin{proof}
  Suppose $k\ge m_{r-1}+m_r$.
If we expand the polynomial in \eqref{ste1} using \eqref{ele33}, then the
terms with highest degree in $z_{r-1}$ are obtained by taking
$\ell_{r-1}=m_{r-1}$ and $\ell_r=m_r$, and ignoring $z_i$ with $i<r-1$ in
$\sumiri\ga_iz_i$; hence these terms are, with $k':=k-m_{r-1}-m_r\ge0$,
\begin{align}
  \sum_{\ell_1+\dots+\ell_{r-2}=k'}
\prod_{i=1}^{r-2}\binom{m_i}{\ell_i} z_i^{\ell_i}\cdot 
\parfrac{-\ga_{r-1}}{\ga_r}^{m_r}z_{r-1}^{m_{r-1}+m_r}.
\end{align}
This gives a nonempty set of monomials with nonzero coefficients, and thus
\eqref{ste1} cannot hold.
\end{proof}

\begin{corollary}\label{C11}
  If $r\ge2$ and $m_{r-1}=m_r=1$, then 
\begin{align}\label{c11}
  \cK=
  \begin{cases}
    \set1, 
\text{ and thus }\cM=\Ls_1(\oi^m),
& \text{if\/ }\ga_i=m_i/m \quad\forall i\in[r],\\
\emptyset, \text{ and thus }\cM=\set0,& \text{otherwise}.
  \end{cases}
\end{align}
\end{corollary}
\begin{proof}
  By \refL{L0} and \refC{C6}, if $k$ is bad, then $k=1$.
The result thus follows from \refC{C1}.
\end{proof}

\begin{corollary}\label{C31}
  If $r\ge3$ and $m_r=1$, then 
\begin{align}\label{c31}
  \cK=
  \begin{cases}
    \set1, 
\text{ and thus }\cM=\Ls_1(\oi^m),
& \text{if\/ }\ga_i=m_i/m \quad\forall i\in[r],\\
\emptyset, \text{ and thus }\cM=\set0,& \text{otherwise}.
  \end{cases}
\end{align}
\end{corollary}

\begin{proof}
  Suppose that $k\in[m]$ is bad.
First, by \refC{C6} (and  permuting $(m_i)_i$ and $(\ga_i)_i$),
we have $k<m_i+1$, and thus $k\le m_i$,  for every $i\le r-1$.
Since $k$ is bad, \refTK{TK6} holds. Let $0\le \ell\le k$, and consider the
terms $z_1^{k-\ell}z_2^\ell$ in \eqref{ste1}; this yields by \eqref{ele33}
(using $m_r=1$)
\begin{align}\label{c32}
  \binom{m_1}{k-\ell}\binom{m_2}{\ell}
-\frac{\ga_1}{\ga_r}
  \binom{m_1}{k-\ell-1}\binom{m_2}{\ell}
-\frac{\ga_2}{\ga_r}
  \binom{m_1}{k-\ell}\binom{m_2}{\ell-1}
=0,
\end{align}
which is equivalent to
\begin{align}\label{c33}
  \ga_r = \ga_1\frac{k-\ell}{m_1-k+\ell+1}
+ \ga_2\frac{\ell}{m_2-\ell+1}.
\end{align}
In particular, for $\ell=0$ and $\ell=k$ we obtain
\begin{align}\label{c34}
  \ga_r = \ga_1\frac{k}{m_1-k+1}
= \ga_2\frac{k}{m_2-k+1}.
\end{align}
Taking a weighted average yields
\begin{align}\label{c35}
  \ga_r &= \frac{k-1}{k}\ga_1\frac{k}{m_1-k+1}
+\frac{1}{k} \ga_2\frac{k}{m_2-k+1}
\notag\\&
= \ga_1\frac{k-1}{m_1-k+1}
+ \ga_2\frac{1}{m_2-k+1},
\end{align}
while taking $\ell=1$ in \eqref{c33} yields
\begin{align}\label{c36}
  \ga_r = \ga_1\frac{k-1}{m_1-k+2}
+ \ga_2\frac{1}{m_2},
\end{align}
which is strictly smaller than \eqref{c35} if $k\ge2$.
This gives a contradiction unless $k=1$, 
so only 1 can possibly be bad.
The result now follows  from \refC{C1}.
\end{proof}

If $r\ge3$, then the non-homogeneous polynomial
$\emmr_k(z_1,\dots,z_{r-1},1)$ obtained by letting $z_r=1$
is a polynomial of degree $k$ in $r-1\ge2$ variables. The condition \refTK{TK5}
implies that this polynomial has a linear factor. This seems to be a strong
requirement; geometrically it means that the $(r-2)$-dimensional surface in
$\bbR^{r-1}$ defined by the polynomial 
(or the corresponding projective surface defined by the homogeneous
polynomial)
is degenerate and contains a
hyperplane. We conjecture that this does not happen.

\begin{conj}\label{Conj:r>2}
  If $r\ge3$, then no $k\ge2$ is bad.
\end{conj}
\refC{Cr=2} and \refEs{Er=2a} and \ref{Er=2} show that this does not hold
for $r=2$.
Note also that 1 may be bad for any $r$, see \refC{C1}.

\section{Are the bad cases really bad?}\label{Sbad}
We continue to assume  \refAss{AAA}.

In \refS{Sgood}, we found many cases where $\tcP_\mmr(F;\gaar)$ is a \qrp,
but also some exceptional cases where the results do not apply.
It is important to realize that these exceptional cases might be an artefact
of our method of proof. In the exceptional cases, the space $\cM$ defined by
\eqref{t1} is nonzero, but this does not necessarily
mean that there exists a non-zero function of the special
form $\tpsifw(\xxm)-p^{e(F)}$ in $\cM$.
Hence it is still possible that 
\eqref{tkx} holds, which by \refT{TKX} shows that
$\tcP_\mmr(F;\gaar)$ is a \qrp.

In fact, the only known counterexamples, when the property
$\tcP_\mmr(F;\gaar)$ is known \emph{not} to be a \qrp,
are the two classical cases noted already in \cite{ChungGW:quasi}
together with a minor extension (mentioned in \cite{SJ294}):
\begin{PXenumerate}{X}
\item\label{cex1} 
$r=1$ and $\ga=1$, and any $F$. (See \refR{Rr=1}.)
\item\label{cex2} 
$r=m=2$, $F=\sK_2$, and $\ga_1=\ga_2=\frac12$. (See \refE{ERK2}.)
\item\label{cex3} 
Some further
cases with $e(F)=1$, so $F=\sK_2$ plus some isolated vertices.
(See \refEs{EK2+}--\ref{EK2+c}.)
\end{PXenumerate}

The most important case when $\cM\neq\set0$ is when
$1$ is the only bad integer; in other words, $\cK=\set1$ and 
thus, by \eqref{cm}, $\cM=\Ls_1(\oi^m)$. We have seen 
in \refC{Cr=m}
that this is the only exceptional case when $r=m$
(and then it occurs when $\ga_i=\frac1m$ for every $i$, see also
\cite[Theorem 2.11]{SJ294}); this case also appears in \refCs{C11} and
\ref{C31}, and in \refConj{Conj:r>2}. 
In this case, we have the following,
which is a simple extension of \cite[Lemma 4.9(b)]{SJ294}:
\begin{lemma}[essentially \cite{SJ294}]\label{LST}
  Suppose that $\cK=\set1$. 
Then $\tcP_\mmr(F;\gaar)$ holds if and only if
there exists a function $h$ on $\oi$ with $\intoi h=p^{e(F)}/m$ such that
\begin{align}\label{lst}
  \tpsifw(\xxm)=\sumim h(x_i)
\quad\text{a.e.\ on $\oi^m$}.
\end{align}

Consequently, 
$\tcP_\mmr(F;\gaar)$ is a \qrp{} if and only \eqref{lst} implies 
that $W$ is constant a.e.
\end{lemma}
\begin{proof}
  This is a simple consequence of \eqref{mac1}, \eqref{cm}, and
the definition of $\Ls_1(\oi^m)$.
\end{proof}

In particular, in the case $r=m$ (so all $m_i=1$), \refC{Cr=m} shows, as
said above, that the only bad case is $\ga_1=\dots=\ga_m=\frac1m$, and then
$\cK=\set1$, so \refL{LST} applies.
This case was studied in \cite{SJ294}, and it was proved 
in \cite[Theorem 2.12]{SJ294} that then
$\tcP_{1,\dots,1}(F;\frac1m,\dots,\frac1m)$ is a \qrp{} 
at least if $F$ is either a regular graph, a star, or disconnected.
We can extend this as follows, for the same graphs but more general
$\mmr$ and $\gaar$ (for examples, see \refCs{C11} and \ref{C31}, and
\refConj{Conj:r>2}). 

\begin{theorem}\label{TB}
  Suppose that
$\cK=\set1$, and that
$F$ is either a regular graph, a star, or disconnected.
Then $\tcPm(F;\gaar)$ is a \qrp.
\end{theorem}
\begin{proof}
By the special case in \cite[Theorem 2.12]{SJ294} just mentioned and
\refL{LST},  for these graphs $F$,
\eqref{lst} implies that $W$ is constant a.e.
Consequently, the result follows by another application of \refL{LST}.
(This is an example of the advantage of splitting the problem of
quasi-randomness into two, as discussed in \refR{Rsplit}.)
\end{proof}

It was conjectured in
\cite[Conjecture 2.13]{SJ294} that 
$\tcP_{1,\dots,1}(F;\frac1m,\dots,\frac1m)$ is a \qrp{}
for every graph $F$ with $e(K)>1$.
The proof of \refT{TB} shows that this is equivalent to the more general:
\begin{conj}\label{CB}
Suppose that
$\cK=\set1$, and that
$e(F)>1$,
Then $\tcPm(F;\gaar)$ is a \qrp.
\end{conj}
See also the corresponding results and conjectures with the non-symmetric
$\cP$ instead of $\tcP$ in 
\cite[Theorem 2.12 and Conjecture 2.13]{SJ294}
and \cite[Theorem 3.6 and Conjecture 6.1]{Hatami2Li}.

If there exists a counterexample to \refConj{CB}, 
then \refL{LST} shows that 
there exists a non-constant graphon $W$ such that \eqref{lst} holds.
It is shown in \cite[Theorem 5.3]{SJ294}, using a removal lemma by
\citet{Petrov}, that in this case,
there exists such a counterexample where $W$
is a 2-type graphon, i.e., a graphon that can be defined on a two-point
probability space.
This leads to an algebraic condition
\cite[Lemma 6.3]{SJ294}, which using \refL{LST} can be formulated as follows.
We let $e_F(A)$ denote the number of edges of $F$ in $A$, and
$e_F(A,A\comp)$ the number of edges between $A$ and $A\comp$.

\begin{lemma}[{Mainly \cite[Lemma 6.3]{SJ294}}]\label{Lalg}
  Suppose that $\cK=\set1$. 
Then the following are equivalent.
\begin{romenumerate}[-10pt]
\item \label{lalg0}
$\tcPm(F;\gaar)$ is \emph{not} \qr.

\item \label{lalguvs}
There exist numbers $u,v,s\ge0$, not all equal, 
and some reals $a$ and $b$, not both $0$, such that
\begin{equation}\label{alg}
  \sum_{A\subseteq V(F):|A|=k} u^{e_F(A)}v^{e_F(A\comp)}s^{e_F(A,A\comp)}
=\binom mk \xpar{a+bk},
\qquad k=0,\dots,m.
\end{equation}
\end{romenumerate}
\end{lemma}

\refConj{CB} is thus equivalent to the conjecture that
\eqref{alg} cannot happen when $e(F)>1$.
This is still an open problem.
\begin{problem}
  Show that if $e(F)>1$, then \eqref{alg} cannot happen, with $u,v,s$ not
  all equal and $a$   and $b$ not both 0.
\end{problem}

We have here discussed the case $\cK=\set1$. As seen in \refC{C11}, there
are also other possibilities of a bad set $\cK$.
It seems natural to guess that also in these cases, if there is a
counterexample, then there is one with a 2-type graphon.
(This would then, for each $\cK$, 
lead to an algebraic condition similar to \eqref{alg}, but
more complicated.)
Unfortunately, it seems difficult to modify the proof 
in \cite[Theorem 5.3]{SJ294} for the case $\cK=\set1$, 
so we leave this as a conjecture.

\begin{conj}
For any $\cK$,
if
$\tcPm(F;\gaar)$ is {not} \qr, then there is a counterexample with a 2-type
graphon. 
\end{conj}

However, as said above, no such counterexamples are known except the ones in
\ref{cex1}--\ref{cex3} above, so we may be more bold. 
\begin{conj}\label{CC}
Suppose that $r>1$ and
$e(F)>1$,
Then $\tcPm(F;\gaar)$ is a \qrp.
\end{conj}

As a concrete example, we have the following, still open, problem from
\cite[Problem 9.5]{SJ294}. In this case we have $\cK=\set2$ by \refE{Er=2a}.

\begin{problem}
  Is $\tcP_{2,1}(\sK_3;\frac13,\frac23)$ a \qrp?
\end{problem}

We end this section with a simple result on changing the test graph $F$ by
adding one or several isolated vertices.
This will provide the counterexamples in \ref{cex3} above.

\begin{theorem}
  \label{T++}
Let $F^*$ be the graph obtained by adding $\ell\ge1$ isolated vertices to
$F$;
thus $m^*:=|F^*|=m+\ell$.
Suppose that
$F^*$, $m^*$, $r^*$, $(m^*_1,\dots,m^*_{r^*})$, and $(\ga^*_1\dots,\ga^*_{r^*})$
also satisfy \refAss{AAA}, 
and let $\cK^*$ be the corresponding set of bad integers.
Then the property
$\tcP_{m^*_1,\dots,m^*_{r^*}}(F^*;\ga^*_1\dots,\ga^*_{r^*})$ is \qr{}
if and only if
\begin{align}\label{t++}
    \tpsifw-p^{e(F)} \in 
\bigoplus_{k\in \cK^*\cap[m]} \Lsmk
\implies
W \text{ is constant a.e.\ on $\oi^2$}.
\end{align}
\end{theorem}
\begin{proof}
It follows from \eqref{psifw} that, keeping the labelling of $F$ also in
$F^*$,
\begin{align}
  \Psi_{F^*,W}(x_1,\dots,x_{m^*})
=\Psi_{F,W}(x_1,\dots,x_{m}).
\end{align}
Hence, it follows from \eqref{symm} and \eqref{Phi} that
\begin{align}
  \tPsi_{F^*,W}=\frac{1}{\binom{m^*}m}\Phi_{m^*,m}(\tpsifw).
\end{align}
Consequently, \refL{LPhi} implies that, using $e(F^*)=e(F)$,
\begin{align}
\tPsi_{F^*,W}-p^{e(F^*)}\in\bigoplus_{k\in \cK^*} \Ls_k\bigpar{\oi^{m^*}}
\iff
\tpsifw-p^{e(F)}\in\bigoplus_{k\in \cK^*\cap[m]} \Ls_k\bigpar{\oi^{m}}.
\end{align}
The result thus follows from \refT{TKX}.
\end{proof}

\begin{example}\label{EK2+}
  Let $e(F)=1$, so $F$ equals $\sK_2$ with $m-2$ isolated vertices added.
It follows from \refT{T++}, changing the notation and  noting that
$\tPsi_{\sK_2,W}(x,y) = \Psi_{\sK_2,W}(x,y)=W(x,y)$,
that
$\tcPm(F;\gaar)$ is a \qrp{} if and only if
\begin{align}\label{t++}
    W-p \in 
\bigoplus_{k\in \cK\cap\set{1,2}} \Lsmk
\implies
W=p \text{ a.e.\ on } \oi^2.
\end{align}
This obviously holds if $\cK\cap\set{1,2}=\emptyset$,
and it is otherwise easy to find graphons $W$ that are counterexamples.
Hence, $\tcPm(F;\gaar)$ is a \qrp{} if and only if $\cK\cap\set{1,2}=\emptyset$.
In particular, for any $m\ge2$, $1\le r\le m$ and $(m_1,\dots,m_r)$ with 
$\sum_jm_j=m$, if we choose $\ga_j=m_j/m$, $j\in[r]$,
then $1\in\cK$ by \refC{C1}, and thus $\tcPm(F;\gaar)$ is not a \qrp.
(There may be other choices of $\ga_j$ that make 2 bad, so this is not if
and only if. 
See \refC{Cr=2} and, in the opposite direction, \refConj{Conj:r>2}.) 
\end{example}

\begin{example}\label{EK2+b}
As a special case of \refE{EK2+}, 
we let $F$ be $\sK_2$ plus two isolated vertices, so
$m=4$. For $r=2$ and $m_1=m_2=2$, it follows from \refE{Er=2} that
$\tcP_{2,2}(F;\ga_1,\ga_2)$ is a \qrp{} if and only if
either $\ga_1+\ga_2<1$, or $\ga_1+\ga_2=1$ and
$\ga_1\notin\set{\frac12,\frac12\pm\frac{\sqrt3}6}$.
Note that this yields counterexamples with $\cK=\set2$ and $\cK=\set{1,3}$.
\end{example}

\begin{example}\label{EK2+c}
As another special case of \refE{EK2+}, 
let $m=r\ge2$, so all $m_i=1$, and let $\ga_i=1/m$ for every $i$.
Then $\cK=\set1$ by \refC{Cr=m}, and thus
\refE{EK2+} shows that if $e(F)=1$, then
$\tcP_{1,\dots,1}(F;\frac1m,\dots,\frac1m)$ is not a \qrp.
(As noted in 
\cite{SJ294}, extending the case $m=r=2$ in
\cite{ChungGW:quasi}, see also \refE{ERK2}).
\end{example}

\section{Further comments}\label{Sfurther}

\subsection{Nonsymmetric version}

In the present paper we consider the symmetrized property $\tcPm(\gaar)$,
and analyze it using a decomposition of the space $\Ls(\oi^m)$ of symmetric
functions into irreducible subspaces.
It would be interesting to find similar results for the nonsymmetrical 
$\cPm(\gaar)$,
using a more general result
for arbitrary subspaces of $L^2(\oi^m)$  without assuming symmetry.
However, we believe that a decomposition into irreducible subspaces in general
is considerably more complicated than in the symmetric case, and therefore
it might be less useful even if it can be found explicitly.
On the positive side, one interesting case of such a decomposition 
is given (somewhat implicitly) by \cite[Theorem 3.1]{Hatami2Li}.

\subsection{Induced copies}

The properties $\cP$ and $\tcP$ studied in the present paper concern counts
of not necessarily induced subgraphs.
It is natural to pose the same questions for counts of induced subgraphs;
such problems were considered by \citet{SS:ind} and \citet{ShapiraYuster}.

Such problems too translate to problems for graphons \cite{SJ234}.
However, a basic problem with induced subgraph counts 
is that (except for the uninteresting cases $e(F)=0$ and $F=\sK_m$),
even if we know that a graphon $W$ is constant, the constant is not uniquely
determined by the function $\psifw^*$ corresponding to $\psifw$,
see \cite{SS:ind}, \cite{ShapiraYuster}, \cite{SJ234}.
Nevertheless, we may still ask whether the graphon $W$ has to be a constant.
Some partial results are given in \cite{SS:ind} and \cite{ShapiraYuster}.

\refT{T1} and the decomposition \eqref{cm} still apply,
but even the case $\cK=\emptyset$ seems to remain open in general.
\begin{problem}
  Prove an analogue of \refT{T0} for counts of induced copies of $F$.
\end{problem}

\section{Proof of \refT{T1}}\label{Spf}

The proof of \refT{T1} 
is based on an approximation with step functions
and an analogue of the theorem for the symmetric
group $\fS_N$ acting on a space of step functions (\refT{TR1}).

\subsection{Step functions}\label{SSstep}

For $N\in\bbN$ and $i\in[N]$,
let $I_{N,i}:=[(i-1)/N,i/N)$; thus $(I_{N,i})_{i=1}^N$ is a partition of
$\oio$. 

For an $m$-tuple $\bi=(i_1,\dots,i_m)\in[N]^m$, let
\begin{align}\label{QN}
Q_{N,\bi}:=\prod_{j=1}^m I_{N,i_j}. 
\end{align}
Thus, $\set{Q_{N,\bi}:\bi\in [N]^m}$ is
a partition of $\oio^m$ into $N^m$ subcubes.
We define the projection $P_N$ in $\Lm$ as the conditional expectation with
respect to this partition, i.e.,
\begin{align}
  \label{d4}
P_Nf(\bx) := N^m\int_{Q_{N,\bi}} f(\by)\dd \by, 
\qquad\text{for }\bx\in Q_{N,\bi}.
\end{align}

$P_N$ is an orthogonal projection of $\Lm$ onto the subspace 
of functions constant on each cube $Q_{N,\bi}$.
(Such functions will be called \emph{step functions}.)
This subspace can be identified with $\ell^2(\Nm)$,
the $N^m$-dimensional complex vector space of
complex-valued functions on $\Nm$, where $\Nm$ is equipped with 
normalized times counting measure.
We therefore denote this space of step functions by
$\tell^2(\Nm)$.
Hence, 
\begin{align}\label{pn}
P_N:\Lm\to\tell^2(\Nm)\subset \Lm,
\end{align}
and there is a natural isometry $\tell^2(\Nm)\cong \ell^2(\Nm)$.

\begin{remark}\label{Rtensor}
Readers familiar with the subject will note that
$\ell^2(\Nm)$ may be regarded as the tensor power
$(\bbC^N)^{\otimes m}$; this is implicit in some arguments below.  
Similarly, $\lls(\Nm)$ defined below equals the symmetric tensor power.
\end{remark}

For each permutation $\gs\in \fS_N$, let $\tgs$ be the map
$\oio\to\oio$ that maps each interval $I_{N,i}$ to $I_{N,\gs(i)}$ by a
translation. 
These maps $\tgs$ are measure-preserving bijections of $\oio$
onto itself, and thus they form a subgroup $\tfS_N$ of $\GG$.
Obviously, $\tfS_N\cong \fS_N$, and we may identify $\tfS_N$
and $\fS_N$.
Hence, we may regard $\rho$ in \eqref{rho}
(restricted to $\gf\in\tfS_n$) 
also as a representation of $\fS_N$ on $L^2(\oi^m)$.
(We may extend $\tgs$ to $\oi$ by $\tgs(1):=1$, but this is obviously 
unimportant, and when convenient we may consider $\oio$
only; this makes no difference since 
functions that are a.e.\ equal are identified in $\Lm$,
and thus $L^2(\oio^m)=L^2(\oi^m)$.)

\begin{remark}
We have $\tfS_N\subseteq \tfS_{nN}$ for every $N,n\in\bbN$, 
and 
it follows that 
$\GG_0:=\bigcup_{N=1}^\infty\tfS_N$ is a (countable) subgroup of $\GG$.
The only measure-preserving bijections that we will  use
are of the type $\tgs$, and thus the proof of \refT{T1} shows that $\rho$ is
irreducible 
on each $\Lsmk$ also as a representation of $\GG_0$.
\end{remark}

The group $\fS_N$ acts (by definition) on $[N]$, and this induces an action
on $\Nm$ by $\gs(i_1,\dots,i_m):=(\gs(i_1),\dots,\gs(i_m))$, which in turn
defines a unitary representation of $\fS_N$ 
on $\ell^2(\Nm)$ by the formula \eqref{rho}, now letting
$x_1,\dots,x_m\in [N]$
(and replacing $\gf$ by $\gs$).
(This is a tensor power of the standard representation on $\bbC^N$, see
\refR{Rtensor}.) 
%
Furthermore, 
it is clear that the space of step functions
$\tell^2(\Nm)$ is an invariant subspace of
the representation $\rho$ in \eqref{rho} restricted to $\tfS_N$;
furthermore, the representation of $\tfS_N$ on $\tell^2(\Nm)$ 
obtained by restricting
$\rho$ obviously agrees with the natural representation of
$\fS_N$ on $\ell^2(\Nm)$
given by the formula \eqref{rho}, 
now with $x_1,\dots,x_m\in [N]$ and $\gf\in\fS_N$.
Hence, the identifications of $\tfS_N$ with $\fS_N$,
and $\tell^2(\Nm)$ with $\ell^2(\Nm)$  cause no problem with the
representations.
We will therefore use the same letter $\rho$
for these  representations 
of $\GG$, $\tfS_N$ and $\fS_N$
on $L^2(\oi^m)$ and $\ell^2(\Nm)$ or various subspaces of them.
(The representation depends on $N$ and $m$, 
and we may write $\rho^{(m)}$ for clarity,
but usually $m$ and $N$ are
clear from
the context and not shown explicitly.)

Recall the notation $\rhof$ in \eqref{hull}.

\begin{lemma}\label{L2}
  For every $f\in \Lm$ and $N\in\bbN$, we have $P_Nf\in\rhof$.
\end{lemma}

We prove first the case $N=1$, in the following more precise form.
Note that $P_1$ is the projection $f\mapsto\int_{\oi^m}f$ onto the constant
functions. 
\begin{lemma}\label{L1}
For $n\in\bbN$, define 
  \begin{align}
    \label{d5}
T_n:=\frac{1}{n!}\sum_{\gs\in \tfS_n}\rho_\gs.
  \end{align}
Then, for every $f\in \Lm$, we have 
as \ntoo,
\begin{align}\label{d6}
  T_n (f) \LLto P_1(f) 
.\end{align}
\end{lemma}
\begin{proof}
$T_n$ is an operator of norm 1 on $\Lm$.
(Actually, it is the projection onto the subspace of functions invariant for
the action of $\tfS_n$, but we do not need this.)

Since \eqref{d6} trivially is true when $f$ is a constant, it suffices to
prove \eqref{d6} for functions $f$ with $P_1f=\int_{\oio^m}f=0$.
Furthermore, since every $T_n$ has norm 1, it suffices to prove \eqref{d6}
for $f$ in a dense subspace. We may therefore assume, 
in addition to $P_1f=0$, that $f$ is  continuous, and that
for some $\gd>0$, we have $f(x_1,\dots,x_m)=0$ whenever $|x_i-x_j|\le\gd$
for some pair of distinct $i,j\in[m]$.

Let $\eps>0$.
Since $f$ is continuous, there exists $\eta>0$ such that
$|f(\bx)-f(\by)|<\eps$ 
when $\bx,\by\in\oio^m$ with 
$|\bx-\by|<\eta$. 
It follows that if $n$ is large enough, then 
$|f(\bx)-f(\by)|<\eps$ whenever $\bx$ and $\by$ belong to the same cube
$Q_{n,\bi}$, and as a consequence, using \eqref{d4}, $|P_nf(x)-f(x)|\le\eps$
for every $\bx\in\oio^m$. In particular, for large $n$,
\begin{align}\label{d7}
  \norm{P_nf-f}_2\le\eps.
\end{align}
Assume also that $n>\gd\qw$. Then, if $\bi\in\Dnm$
and $\bx=(x_1,\dots,x_m)\in Q_{n,\bi}$, we have $i_j=i_k$ for some pair
$j<k$ and thus $|x_j-x_k|<1/n<\gd$; hence, by our assumption,
$f(\bx)=0$ on $Q_{n,\bi}$. 
Consequently, $f=0$ on  every diagonal cube,
and thus also $P_nf=0$ there.

Consider now $T_n(P_nf)$, and note first that this function, 
as $P_nf$, is constant on every cube $Q_{n,\bi}$. 
Since every $\gs\in\fS_n$ permutes the diagonal cubes 
$Q_{n,\bi}$, it follows from \eqref{d5} that $T_n(P_nf)=0$ on every diagonal cube.
Furthermore, also the off-diagonal cubes are permuted; moreover, 
when $\fS_n$ acts on the set of indices $\nm$, 
then the number of permutations
$\gs\in\fS_n$ mapping a given off-diagonal index onto another is the constant
$n!/m!$. Consequently, it follows from \eqref{d5} that 
the value of $T_n(P_nf)$ on an
off-diagonal cube is the average of $P_nf$ over all off-diagonal cubes.
Since $P_nf$ has average 0 and is 0 on the diagonal cubes, this yields
$T_n(P_nf)=0$ on every off-diagonal cube.
Consequently, $T_n(P_nf)=0$, provided $n$ is large enough.

Hence, for large $n$,
\begin{align}
\normm{ T_n(f)-P_1f}
=\normm{ T_n(f)}
=\normm{ T_n(f)-T_n(P_nf)}
\le \normm{f-P_nf} \le\eps.
\end{align}
This proves that $T_nf\to P_1f$ for $f$ of the special type considered in
the proof, and thus as said above for all $f\in \Lm$; this verifies \eqref{d6}
and completes the proof.
\end{proof}

\begin{proof}[Proof of \refL{L2}]
  Let $n\in\bbN$, and consider the subgroup $G_{N,n}$ of $\fS_{nN}$ that
preserves the $N$ subsets $\set{1,\dots,n},\dots,\set{(N-1)n+1,\dots,Nn}$
of $[nN]$.
This group is obviously and naturally isomorphic to $\fS_n^N$.
Furthermore,  the corresponding subgroup $\tG_{N,n}$ of $\tfS_{nN}$
preserves every interval $I_{N,i}$ of length $1/N$.
We now define, in analogy with \eqref{d5},
  \begin{align}
    \label{d8}
T_{N,n}:=\frac{1}{n!^N}\sum_{\gs\in \tG_{N,n}}\rho_\gs.
  \end{align}
Again, this is an operator on $\Lm$ of norm 1.
We claim that, for every $f\in \Lm$,
\begin{align}  \label{d9}
T_{N,n}(f)\LLto P_N(f),
\qquad\text{as \ntoo}.
\end{align}
To see this, note first that 
each $\gs\in \tG_{N,n}$ preserves each cube $Q_{N,\bi}$,
so we can study these cubes individually. 
If we map $Q_{N,\bi}$ onto the entire cube $\oio$ by the obvious linear map,
then the action of $T_{N,n}$ on $L^2(Q_{N,\bi})$ corresponds to the operator
\begin{align}\label{dd1}
  \Tnni g(x_1,\dots,x_m)
=\frac{1}{n!^N} \sum_{\gs_1,\dots,\gs_N\in\tfS_n}  
 g(\gs_{i_1}(x_1),\dots,\gs_{i_m}(x_m))
\end{align}
acting on $g\in L^2(\oio^m)=\Lm$.

This action on 
the cube $Q_{N,\bi}$ depends on the index $\bi=(i_1,\dots,i_m)$; 
more precisely, it depends
on which equalities $i_j=i_k$ there are among the indices. 
Consider a case where there are $\ell$ different values among
$\set{i_1,\dots,i_m}$; we may without loss of generality assume that these
are $1,\dots,\ell$. Let $J_k:=\set{j:i_j=k}$ for $k\in[\ell]$.
Suppose that $g$ is of the special form
\begin{align}\label{hhh}
  g(x_1,\dots,x_m)=g_1(x_{J_1}) \dotsm g_\ell(x_{J_\ell}),
\end{align}
where $g_k\in L^2(\oi^{|J_k|})$ for $k\in[\ell]$.
(Recall the notation \eqref{da3}, and note that this is a more general
version of \eqref{hh}.)
Then \eqref{dd1} yields, using the notation \eqref{d5},
\begin{align}\label{thhh}
  \Tnni g(x_1,\dots,x_m)
=
T_{n}g_1(x_{J_1})\dotsm T_{n}g_\ell(x_{J_\ell})
.\end{align}
\refL{L1} shows that, as $n\to\infty$,
each $T_ng_k\LLto P_1 g_k$, and thus if follows from \eqref{thhh} that
\begin{align}\label{thx}
  \Tnni g(x_1,\dots,x_m)
\LLto
P_1g_1(x_{J_1}) \dotsm P_1g_\ell(x_{J_\ell})
=\prod_{k=1}^\ell \int_{\oi^{|J_k|}} g_k
=\int_{\oi^m}g,
\end{align}
for every $g\in \Lm$ of the special type \eqref{hhh}. 
Linear combinations of such functions are dense in $\Lm$, as is seen, for
example, by the Fourier series expansion. Since $T_{N,n,\bi}$ has norm 1, it
follows that
\begin{align}\label{thy}
  \Tnni g(x_1,\dots,x_m)
\LLto
\int_{\oi^m}g
\end{align}
for every $g\in\Lm$.

Returning to the cube $Q_{N,\bi}$, \eqref{thy} shows that
$T_{N,n}f \to N^m\int_{Q_{N,\bi}}f=P_Nf$ in $L^2(Q_{N,\bi})$,
for every $\bi\in\Nm$, which yields \eqref{d9}.
This completes the proof, since evidently $T_{N,n}(f)\in\rhof$.
\end{proof}

\subsection{Spaces of step functions}\label{SSdisc2}
Let $N\ge1$ and $m\ge0$.
We only need the results for large $N$;
we ignore complications that occur for small $N$ by making explicit assumptions
on $N$ when needed. (We invite the curious reader to investigate the
exceptional cases with small $N$.)
Recall that we may identify the 
$N^m$-dimensional complex vector space $\ell^2(\Nm)$  of
complex-valued functions on $\Nm$
and the corresponding space $\tell^2(\Nm)\subset L^2(\oi^m)$ 
of step functions constant on every cube $Q_{N,\bi}$.
We define some subspaces of $\ell^2(\Nm)$ that will be important;
these can be identified with the corresponding subspace of
$\tell^2(\Nm)\subset L^2(\oi^m)$.

First, we define  the set of \emph{diagonal} indices
\begin{align}\label{di1}
 \DNm:=\bigset{(i_1,\dots,i_m)\in\Nm: i_k=i_\ell \text{ for some pair $k<\ell$}},
\end{align}
its complement,
the \emph{off-diagonal} indices,
\begin{align}
\DcNm:=\Nm\setminus \DNm
=\bigset{(i_1,\dots,i_m)\in\Nm: i_1,\dots,i_m \text{ are distinct}}.
\end{align}
We will later also use (for functions in $\tell(\Nm)$) 
the corresponding subsets of $\oio^m$:
\begin{align}\label{di3}
 \cDNm&:=\bigcup_{\bi\in \DNm} Q_{N,\bi} ,
\\\label{di4}
 \cDcNm&:=\bigcup_{\bi\notin \DNm} Q_{N,\bi} = \oio^m\setminus \cD_N,
\end{align}

We will mainly be interested in 
symmetric
functions in $\ell^2(\Nm)$ that vanish on the
set $\DNm$ of diagonal indices defined in \eqref{di1}. 
We assume in the sequel $N\ge m$ (so that $\DcNm\neq\emptyset$).
We use the natural identification
\begin{align}
  \label{lo0}
\ell^{2}(\DcNm)&=\bigset{f\in\ell^2(\Nm):f(\bi)=0 \text{ for }\bi\in \DNm}.
\end{align}
We further define
\begin{align}
\ell^{2,\mathsf s}(\Nm)&:=
\text{the subspace of symmetric functions in } \ell^2(\Nm),\label{lo00}
\\
\lls(\DcNm)&:=\ell^{2}(\DcNm)\cap \ell^{2,\mathsf s}(\Nm)\label{lo1}
.\end{align}

Recall that
$\ZZ Nm$ denotes the set of all $\binom Nm$ subsets of ${N}$ of size $m$.
We can identify the space of symmetric functions on $\DcNm$
with the space of functions on $\ZZ Nm$; thus
\begin{align}\label{llD}
\lls(\DcNm) =\ell^2\bigpar{\tZZ Nm},
\end{align}
where we
normalize the counting measure on $\ZZ Nm$ such that
\eqref{llD} is an isometry.
In particular,
\begin{align}\label{a2}
  \dim(\lls(\DcNm)) = \tbinom Nm.
\end{align}
Let $\nu_{N,m}$ denote the counting measure on $\ZZ Nm$
with our normalization used in \eqref{llD}; 
then its total mass is
\begin{align}\label{ste3}
  \nu_{N,m}\bigpar{\tZZ Nm} = N^{-m}|\DcNm| = \frac{N!/(N-m)!}{N^m}.
\end{align}
Note that for a fixed $m$, this increases to 1 as $\Ntoo$, and that for 
every $N\ge m$ we have
\begin{align}\label{ste4}
  m!/m^m \le \nu_{N,m}\bigpar{\tZZ Nm} < 1.
\end{align}

For $0\le k\le m$, 
define a linear map
$\PhiN_{m,k}:\lls(\DcNk)\to \lls(\DcNm)$ 
by,
for distinct $x_1,\dots,x_m\in[N]$,
\begin{align}\label{PhiN}
(\PhiN_{m,k}f)(x_1,\dots,x_m):=\sum_{I\in \ZZ mk} f(x_I),
\end{align}
where we 
recall the notation \eqref{da3}, and
sum over the $\binom mk$ subsets of $[m]$ of size $k$.
Equivalently, we can regard $\PhiN_{m,k}$ as  a linear map
$\ell^2\bigpar{\ZZ Nk}\to\ell^2\bigpar{\ZZ Nm}$
defined by, for $I\in\ZZ Nm$,
\begin{align}\label{Phix1}
(\PhiN_{m,k}f)(I):=\sum_{J\subseteq I: |J|=k} f(J).
\end{align}

Denote the range of $\PhiN_{m,k}$ by
\begin{align}\label{a5}
\llslek(\DcNm)
:=\PhiN_{m,k}(\lls(\DcNk))
\subseteq \lls(\DcNm).
\end{align}

If $0\le k\le \ell\le m$ and $f\in \lls(\DcNk)$, then,
for $(x_1,\dots,x_m)\in \DcNm$,
\begin{align}\label{a6}
&  \PhiN_{m,\ell}\PhiN_{\ell,k} f(x_1,\dots,x_m)
=\sum_{I\in \ZZ m\ell} \PhiN_{\ell,k}f(x_I)
=\sum_{I\in \ZZ m\ell,\; J\in \ZZq Ik} f(x_J)
\notag\\&\qquad
=\sum_{J\in \ZZ mk} \binom{m-k}{\ell-k}f(x_J)
=\binom{m-k}{\ell-k}\PhiN_{m,k}f(x_1,\dots,x_m).
\end{align}
Thus,
$\PhiN_{m,\ell}\PhiN_{\ell,k}=\binom{m-k}{\ell-k}\PhiN_{m,k}$,
and it follows that, for $0\le k\le \ell\le m$,
\begin{align}\label{a7}
\PhiN_{m,\ell}\bigpar{\llslek(\DcNx\ell)}
=\PhiN_{m,\ell}\PhiN_{\ell,k}\bigpar{\lls(\DcNk)}
=\PhiN_{m,k}\bigpar{\lls(\DcNk)}
=\llslek(\DcNm).  
\end{align}

\begin{lemma}\label{LR0}
Let $0\le k\le m$ and assume $N\ge m+k$. Then the map 
$\PhiN_{m,k}:\lls(\DcNk)\to \lls(\DcNm)$ is injective.
Hence, $\PhiN_{m,k}$ is a bijection
$\lls(\DcNk)\to \llslek(\DcNm)$.
Furthermore, there is a constant $C_{m,k}$ (not depending on $N$) such that
for every $f\in\lls(\DcNk)$,
\begin{align}
  \label{lr0}
\normm{f}\le C_{m,k} \normm{\PhiN_{m,k}f}.
\end{align}
\end{lemma}

It follows from \eqref{a2} that 
when $N<m+k$,
$\PhiN_{m,k}$ is \emph{not} injective 
(and thus \eqref{lr0} does not hold)
except in the trivial cases $k=m$ or $N<k$.

\begin{proof}
In this proof it is convenient to regard $\PhiN_{m,k}$ as the map
$\ell^2(\ZZ Nk)\to\ell^2(\ZZ Nm)$ given by \eqref{Phix1}.

Let $f\in \ell^2(\ZZ Nk)$.
  Fix two disjoint sets $A,B\subseteq[N]$ with $|A|=k$ and $|B|=m$.
For $j=0,\dots,k $,
let
\begin{align}\label{lr01}
  y_j&:=\sum_{I\in\ZZq Aj,\,J\in\ZZq B{k-j}} f(I\cup J),
\\\label{lr02}
  z_j&:=\sum_{I\in\ZZq Aj,\,J\in\ZZq B{m-j}} \PhiN_{m,k}f(I\cup J).
\end{align}
Then, by \eqref{Phix1} and interchanging the order of summation,
\begin{align}\label{lr03}
    z_j&=\sum_{I\in\ZZq Aj,\,J\in\ZZq B{m-j}} 
\sum_{i=0}^j
\sum_{I'\in\ZZq Ii,\,J'\in\ZZq J{k-i}} f(I'\cup J')
\notag\\&
=\sum_{i=0}^j
\sum_{I'\in\ZZq Ai,\,J'\in\ZZq B{k-i}} 
\binom{k-i}{j-i}\binom{m-k+i}{m-j-k+i}
f(I'\cup J')
\notag\\&
= \sum_{i=0}^j\binom{k-i}{j-i}\binom{m-k+i}{m-j-k+i}y_i.
\end{align}
Now suppose that $\PhiN_{m,k}f=0$.  
Then all $z_j=0$, and \eqref{lr03} yields a triangular
system of equations for $y_0,\dots,y_k$, with strictly positive
coefficients on the diagonal. 
Hence this system has the unique solution $y_0=\dots = y_k=0$.
Moreover, \eqref{lr01} yields $y_k=f(A)$, and thus $f(A)=0$.
Since $A$ is arbitrary, we see that $f=0$, and thus $\PhiN_{m,k}$ is injective.
Hence,
$\PhiN_{m,k}$ is a bijection of $\llll(\ZZ Nk)$ onto its image,
which by the identification \eqref{llD} and
\eqref{a5} means that $\PhiN_{m,k}$ is a bijection of
$\lls(\DcNk)$ onto $\llslek(\DcNm)$.

In particular, taking $N=m+k$, $\PhiN_{m,k}$ is a bijection of
$\llll\bigpar{\ZZ{m+k}k}$ onto a subset of
$\llll\bigpar{\ZZ{m+k}m}$.
Since these spaces are
finite-dimensional, it follows that there is a constant $C_{m,k}$ such that,
for all $f\in\llll\bigpar{\ZZ{m+k}k}$,
\begin{align}\label{lr04}
  \frac{1}{\binom{m+k}k} \sum_{I\in\ZZ{m+k}{k}}|f(I)|^2
=\normm{f}^2
\le C_{m,k}^2\normm{\PhiN_{m,k}f}
=  \frac{C_{m,k}^2}{\binom{m+k}m} \sum_{J\in\ZZ{m+k}{m}}|\PhiN_{m,k}f(J)|^2.
\end{align}
Now, let $N\ge m+k$ be arbitrary, and take a set $X\in\ZZ N{m+k}$.
It follows from the definition \eqref{Phix1} that the restriction 
of $\PhiN_{m,k}f$ to $\ZZq Xm$ depends only on the restriction of $f$ to
$\ZZq Xk$; furthermore, this restriction defines a map 
$\ell^2\bigpar{\ZZq Xk}\to \ell^2\bigpar{\ZZq Xm}$ which up to notational
differences is the same as 
$\PhiN_{m,k}:\ell^2\bigpar{\ZZ {m+k}k}\to \ell^2\bigpar{\ZZ {m+k}m}$.
Hence, \eqref{lr04} yields
\begin{align}\label{lr05}
  \frac{1}{\binom{m+k}k} \sum_{I\in\ZZq{X}{k}}|f(I)|^2
\le C_{m,k}^2 \frac{1}{\binom{m+k}m} \sum_{J\in\ZZq{X}{m}}|\PhiN_{m,k}f(J)|^2.
\end{align}
The \lhs{} is the average of $|f(I)|^2$ over all $I\in\ZZq Xk$, and
similarly for the \rhs.
Taking the average of \eqref{lr05} over all $X\in\ZZ N{m+k}$ yields the
averages over all $I\in\ZZ Nk$ and $J\in\ZZ Nm$, i.e.,
\begin{align}\label{lr06}
  \frac{1}{\binom{N}k} \sum_{I\in\ZZ{N}{k}}|f(I)|^2
\le C_{m,k}^2 \frac{1}{\binom{N}m} \sum_{J\in\ZZ{N}{m}}|\PhiN_{m,k}f(J)|^2.
\end{align}
This is \eqref{lr0} if we define the norms using counting measures
normalized to have total mass 1. We use instead the normalization \eqref{ste3},
but \eqref{ste4} shows 
that this changes the norms by at most a constant factor,
not depending on $N$. Hence, changing $C_{m,k}$, 
\eqref{lr0} is valid also for our norms.
\end{proof}

It follows from \eqref{a7} and \eqref{a5} that if $0\le k\le\ell$, then
\begin{align}\label{ste11}
\llslek(\DcNm)
=
\PhiN_{m,\ell}\bigpar{\llslek(\DcNx\ell)}
\subseteq
\PhiN_{m,\ell}\bigpar{\lls(\DcNx\ell)}
=
\llslex{\ell}(\DcNm).
\end{align}
Hence,
defining also $\llslex{-1}(\DcNm):=\set0$,
\begin{align}  \label{d2sN}
\set0
=
\llslex{-1}(\DcNm)
\subseteq
\llslex{0}(\DcNm)
\subseteq\dotsm\subseteq
\llslex{m-1}(\DcNm)
\subseteq \llslex{m}(\DcNm)
= \lls(\DcNm).
\end{align}
We define, analoguously to \refSS{SS1}, the orthogonal complemets
\begin{align}\label{ste12}
  \lls_k(\DcNm)
:=\llslex{k}(\DcNm)
\ominus 
\llslex{k-1}(\DcNm).
\end{align}
Then
$\lls_k(\DcNm)$, $k=0,\dots,m$, are pairwise orthogonal subspaces of
$\lls(\DcNm)$, and
\begin{align}\label{a10}
  \llslex{\ell}(\DcNm)=\bigoplus_{k=0}^\ell \lls_k(\DcNm),
\qquad 0\le \ell\le m.
\end{align}

\subsection{Some representations of the symmetric group}
\label{SSrep}

In the present subsection, we consider $\rho$ as a representation of $\fS_N$
on various spaces.
We  use the notation
$\rho\restr{V}$  when we want to emphasize that we consider
$\rho$ as a representation on a subspace $V\subseteq \ell^2(\Nm)$.

We will use some basic facts from the theory of finite-dimensional
representation of finite groups; see \refApp{Arep}
or, for example, \cite{FultonHarris} and \cite{Serre}.

\begin{remark}
  There is a large theory describing the representations of $\fS_N$, see for
  example \cite[Lecture 4]{FultonHarris} for some parts of it.
We avoid using any deep results here, and will only use basic representation
theory and some simple combinatorial calculations with characters.
The results below are presumably all known, but we do not
know any references and we give full proofs for completeness.
\end{remark}

The subspaces $\ell^{2}(\DcNm)$ and $\lls(\DcNm)$
defined in \eqref{lo0} and \eqref{lo1}
are clearly
invariant under $\rho=\rho^{(m)}$.
Furthermore, the map $\PhiN_{m,k}$ in \eqref{PhiN} intertwines the action of
$\fS_N$ on $\lls(\DcNk)$ and $\lls(\DcNm)$, i.e.,
\begin{align}\label{tw1}
  \PhiN_{m,k}(\rho^{(k)}_\gs f) = \rho^{(m)}_\gs(\PhiN_{m,k} f),
\qquad f\in \lls(\DcNk),\; \gs\in \fS_N.
\end{align}
It follows that also 
$\llslek(\DcNm)$ and $\lls_k(\DcNm)$  defined in
\eqref{a5} and \eqref{ste12}
are invariant under $\rho=\rho^{(m)}$.
Consequently, $\rho$ also yields representations of $\fS_N$
on these four spaces.

We denote the characters of the representation $\rho$ 
on some spaces by,
for $0\le k\le m$,
\begin{align}\label{lr00a}
  \chi_m:=\chi_{\rho\restr{\lls(\DcNm)}},
\qquad
  \chimle{k}:=\chi_{\rho\restr{\llslek(\DcNm)}},
\qquad
  \chi_{m,k}:=\chi_{\rho\restr{\lls_k(\DcNm)}}.
\end{align}
Recall that $\llslex{m}(\DcNm)=\lls(\DcNm)$ and thus $\chimle{m}=\chi_m$.
Furthermore:
\begin{lemma}\label{LR00}
  Let $0\le k\le m$ and assume $N\ge k+m$. Then the
representations $\rho\restr{\lls(\DcNk)}$ and $\rho\restr{\llslek(\DcNm)}$
of $\fS_N$
are isomorphic, 
and thus $\chimle{k}=\chi_k$.
\end{lemma}
\begin{proof}
This follows from \refL{LR0} and \eqref{tw1}, 
which show that  $\PhiN_{m,k}$ yields the desired isomorphism of the
representations.
\end{proof}

It follows from \eqref{a10}
that $\llslek(\DcNm)=\llslex{k-1}(\DcNm)\oplus \lls_k(\DcNm)$ and thus,
using the notation \eqref{lr00a},
$\chimle{k}=\chimle{k-1}+\chi_{m,k}$  (with $\chimle{-1}:=0$). Hence,
\begin{align}\label{b5a}
  \chi_{m,k}=\chimle{k}-\chimle{k-1},\qquad 0\le k\le m.
\end{align}
Consequently, \refL{LR00} implies (with $\chi_{-1}:=0$)
\begin{align}\label{b5b}
  \chi_{m,k}=\chi_k-\chi_{k-1}=:\chix_k,\qquad 0\le k\le m,\; N\ge k+m.
\end{align}

\begin{theorem}\label{TR1}
  Assume\/ $N\ge 2m$. Then the representation $\rho\restr{\lls_k(\DcNm)}$ 
of $\fS_N$
is irreducible
  for $k=0,\dots,m$, and these representations are distinct 
(i.e., non-isomorphic). Hence, \eqref{a10} yields the (unique) decomposition of
  $\llslek(\DcNm)$ into its irreducible components.
\end{theorem}

\begin{proof}
Consider first $\lls(\DcNm)$.
As said in \eqref{llD}, we can identify $\lls(\DcNm)$ with the
space $\ell^2(\ZZ Nm)$ of all functions $\ZZ{N}m\to\bbC$, and since a
permutation $\gs\in \fS_N$ 
acts as a permutation on $\ZZ{N}m$ in the obvious way, it follows 
(see \refE{Echi})
that $\chi_{m}(\gs)$ equals
the number of subsets $A\subseteq N$ of size $m$ that are fixed by $\gs$,
i.e.,
\begin{align}\label{b2.5}
  \chi_m(\gs)=
\sum_{A\in\ZZ{N}m}\indic{\gs(A)=A}.
\end{align}
Hence, for any $k$ and $m\ge0$, 
\begin{align}\label{b3}
  \innprod{\chi_{k},\chi_{m}}&
:=\frac{1}{N!}\sum_{\rho\in \fS_N}\chi_{k}(\gs)
\overline{\chi_{m}(\gs)}
\notag\\&
=\frac{1}{N!}
\sum_{\gs\in \fS_N}\sum_{A\in\ZZ{N}k,\,B\in\ZZ{N}m}\indic{\gs(A)=A}\indic{\gs(B)=B}. 
\end{align}
We interchange the order of summation, and note that the product of
indicator functions is 1 if and only if $\gs$ maps the four disjoint sets
$A\cap B$, $A\setminus B$, $B\setminus A$, $[N]\setminus(A\cup B)$ into
themselves. Hence, if $|A\cap B|=j\le k\bmin m$, 
then the number of such $\gs$ is 
$j!\,(k-j)!\,(m-j)!\,(N-m-k+j)!$. Furthermore, for every $j\ge0$ with 
$j\le k$, $j\le m$, and $m+k-j\le N$, the number of such pairs $(A,B)$ of
sets is $\binom{N}{j,k-j,m-j,N-m-k+j}$.
Consequently, if we  assume $k\le m$,
then, since $N\ge2m\ge m+k$,
\begin{align}\label{b4}
&  \innprod{\chi_{k},\chi_{m}}
\notag\\&
=\frac{1}{N!}\sum_{j=0}^k\binom{N}{j,k-j,m-j,N-m-k+j}
j!\,(k-j)!\,(m-j)!\,(N-m-k+j)!
\notag\\&
=\sum_{j=0}^k1
=k+1.
\end{align}

It follows from \eqref{b4} and \eqref{b5b} that if $k,\ell\le m$, 
then
\begin{align}
  \innprod{\chix_k,\chix_\ell}&
=  \innprod{\chi_k,\chi_\ell} -   \innprod{\chi_k,\chi_{\ell-1}}
-   \innprod{\chi_{k-1},\chi_\ell}
+   \innprod{\chi_{k-1},\chi_{\ell-1}}
\notag\\&=
(k+1)\land(\ell+1) - (k+1)\land\ell-k\land(\ell+1)+k\land\ell
\notag\\&
=
  \begin{cases}
    0, &k\neq\ell,
\\
1, & k=\ell.
  \end{cases}
\end{align}
Hence, the characters $\chix_0,\dots,\chix_m$ are orthonormal, which shows
(using \eqref{b5b} again and \eqref{lr00a})
that the corresponding representations 
$\rho\restr{\lls_0(\DcNm)}, \dots,\rho\restr{\lls_m(\DcNm)}$ 
are irreducible and distinct.
(See \refApp{Arep}.)
\end{proof}

\begin{remark}\label{RYoung}
  It is well-known that the irreducible representations of $\fS_N$ are in 1--1
  correspondence with Young diagrams of size $N$, see for example
\cite[Lecture 4]{FultonHarris}. It can be shown that, for $N\ge m+k$,
the representation $\rho\restr{\lls_k(\DcNm)}$ corresponds to the Young diagram
$(N-k,k)$, see \refE{EYoung}.
\end{remark}

\begin{theorem}\label{TR2}
Assume $N\ge 2m$. Then the irreducible representation
$\rho\restr{\lls_m(\DcNm)}$ 
of $\fS_N$
appears exactly once as a component of $\rho\restr{\ell^2(\Nm)}$, and it
does not appear at all as a component of $\rho\restr{\ell^2(\Nk)}$ for any $k<m$.
\end{theorem}
\begin{proof}
  Let us first compute the character $\chi_{\rho\restr{\ell^2(\Nk)}}$.
For $\gs\in \fS_N$, let $\tau(\gs)$ be the number of fixed points of $\gs$.
Since $\ell^2(\Nk)$ is the space of all functions on $[N]^k$, and $\fS_N$ acts by
permuting the elements of $[N]^k$, the character is 
(see again \refE{Echi}) 
given by the number of
fixed points of this action, which is $\tau(\gs)^k$. 
Thus
\begin{align}
  \chi_{\rho\restr{\ell^2(\Nk)}}(\gs)=\tau(\gs)^k,
\qquad \gs\in \fS_N
.\end{align}

We  compute the mean of this character, which we denote by
$M_k:=\innprod{\tau^k,1}$.
It is simpler to consider the descending factorials, denoted by
$(a)_\ell:=a(a-1)\dotsm(a-\ell+1)$. Then, 
for $\ell\le N$, 
with $\sumx$ denoting the sum over
distinct variables $x_1,\dots,x_\ell$ only and
by interchanging the order of summation,
\begin{align}\label{c2}
  \innprod{(\tau)_\ell,1}&
=\frac{1}{N!}\sum_{\gs\in  \fS_N}(\tau(\gs))_\ell
=\frac{1}{N!}\sum_{\gs\in  \fS_N}
\sumx_{x_1,\dots,x_\ell\in[N]}\indic{\gs(x_i)=x_i\;\forall i}
\notag\\&
=\frac{1}{N!}
\sumx_{x_1,\dots,x_\ell\in[N]}(N-\ell)!
=1.
\end{align}
(See also \refR{Rpoisson} below.)
If $\ell>N$, we instead trivially have $(\tau)_\ell=0$. Hence we conclude
\begin{align}\label{c22}
  \innprod{(\tau)_\ell,1}&
=
\begin{cases}
1, & \ell \le N,
\\
0, & \ell>N.                        
\end{cases}
\end{align}
We have $x^k=\sum_{\ell=0}^k \stirling {k}{\ell} (x)_\ell$, 
where $\stirling n\ell$ are the 
Stirling numbers of the second kind,  
see e.g.\ \cite[(6.10)]{CM} or \cite[1.(24d)]{StanleyI}. 
Hence, \eqref{c22} yields
\begin{align}\label{c1a}
M_k:=\innprod{\tau^k,1}
=\frac{1}{N!}\sum_{\gs\in \fS_N}\tau(\gs)^k
=\sum_{\ell=0}^k \stirling {k}{\ell}\innprod{(\tau)_\ell,1} 
=\sum_{\ell=0}^{k\bmin N} \stirling {k}{\ell}.
\end{align}
In particular, for $k\le N$,
\begin{align}\label{c1b}
M_k=\sum_{\ell=0}^{k} \stirling {k}{\ell}
=:B_k,
\qquad 0\le k\le N,
\end{align}
where $B_k$ is the $k$:th Bell number
\cite[Exercise 7.15]{CM},
see again \refR{Rpoisson}.
(We can take \eqref{c1b} as a definition
of $B_k$ and do not need to know its combinatorial meaning.)

Furthermore, for the case $k=N+1$,
 \eqref{c1a} yields
\begin{align}\label{c1c}
M_k=\sum_{\ell=0}^{k-1} \stirling {k}{\ell}
=B_k - \stirling{k}k
=B_k-1,
\qquad k=N+1.
\end{align}

Next, we consider $\innprod{\tau^k,\chi_m}$. We have, by \eqref{b2.5},
\begin{align}\label{c3}
  \innprod{\tau^k,\chi_m}
=\frac{1}{N!}\sum_{\gs\in \fS_N} \sum_{A\in\ZZ{N}m}\indic{\gs(A)=A} 
\tau(\gs)^k
\end{align}
For a given $A$, the permutations $\gs\in \fS_N$ that satisfy $\gs(A)=A$ 
consist of one permutation of $A$ and another of $[N]\setminus A$. 
Hence \eqref{c3} yields, by interchanging the order of summation,
\begin{align}\label{c4}
  \innprod{\tau^k,\chi_m}&
=\frac{1}{N!} \sum_{A\in\ZZ{N}m}\sum_{\gs_1\in \fS_m}\sum_{\gs_2\in \fS_{N-m}}
(\tau(\gs_1)+\tau(\gs_2))^k
\notag\\&
=\frac{1}{N!}\binom{N}{m}
\sum_{\gs_1\in \fS_m}\sum_{\gs_2\in \fS_{N-m}}
 \sum_{j=0}^k\binom kj\tau(\gs_1)^j\tau(\gs_2)^{k-j}
\notag\\&
= \sum_{j=0}^k\binom kj
\frac{1}{m!}\sum_{\gs_1\in \fS_m}\tau(\gs_1)^j
\frac{1}{(N-m)!}\sum_{\gs_2\in \fS_{N-m}}\tau(\gs_2)^{k-j}
.\end{align}
If $k\le m$ and $k\le N-m$, then
the two normalized inner sums on the last line of \eqref{c4}
are evaluated by \eqref{c1a}--\eqref{c1b} to $B_j$ and $B_{k-j}$, respectively.
Hence, \eqref{c4} yields
\begin{align}\label{c5}
  \innprod{\tau^k,\chi_m}&
= \sum_{j=0}^k\binom kj B_j B_{k-j},
\qquad 0\le k\le m, \;  N\ge k+m.
\end{align}
Note that the \rhs{} does not depend on $m$, subject to $k\le m\le N-k$.
Consequently, if $0\le k<m$ and $N\ge 2m$, 
so that \eqref{c5} applies both for $m$ and for $m-1$, 
then \eqref{b5b} and \eqref{c5}
yield
\begin{align}\label{c6}
    \innprod{\tau^k,\chix_m}&
=  \innprod{\tau^k,\chi_m} -   \innprod{\tau^k,\chi_{m-1}}
=0.
\end{align}
This yields the claim for $k<m$, since $\tau^k$ is the character of
$\rho\restr{\ell^2(\Nk)}$ 
and $\chix_m$ is the character of $\rho\restr{\lls_m(\DcNm)}$, which is irreducible by
\refT{TR1}.

If $k=m+1$ and $N\ge k+m=2k-1$, we obtain similarly from \eqref{c4},
\eqref{c1b}, and \eqref{c1c} (the latter for the case $j=k$),
using also $B_0=1$ (see \eqref{c1b}),
\begin{align}\label{c7}
  \innprod{\tau^k,\chi_{k-1}}&
= \sum_{j=0}^{k-1}\binom kj B_j B_{k-j} + \binom kk (B_k-1)B_0
= \sum_{j=0}^{k}\binom kj B_j B_{k-j} -1
.\end{align}
Consequently, assuming $N\ge 2m$, for $k=m$ 
we obtain instead of \eqref{c6},
from \eqref{c5} and \eqref{c7} (with $k$ replaced by $m$),
\begin{align}\label{c8}
    \innprod{\tau^m,\chix_m}&
=  \innprod{\tau^m,\chi_m} -   \innprod{\tau^m,\chi_{m-1}}
=1.
\end{align}
This shows that  the irreducible representation $\rho\restr{\lls_m(\DcNm)}$
appears exactly once in $\rho\restr{\ell^2(\Nm)}$, as claimed.
\end{proof}

\begin{remark}\label{Rpoisson}
 The formulas \eqref{c2} and \eqref{c1b} are well-known, and say
that the number of fixed points in a random permutation of $[N]$ has 
the same factorial moments and  moments up to order $N$ as a 
Poisson$(1)$ distribution; 
this implies the even more well-known fact that 
the distribution of
the number of fixed points is asymptotically Poisson$(1)$ as $\Ntoo$. 
\end{remark}

We already know that  $\rho\restr{\lls_m(\DcNm)}$
appears in $\rho\restr{\ell^2(\Nm)}$, since $\lls_m(\DcNm) \subseteq \ell^2(\Nm)$. The point of
\refT{TR2} is that this is the only way it appears in $\ell^2(\Nm)$.

\begin{corollary}\label{CR1}
  Let $m\ge1$ and $N\ge 2m$. Then the operator
\begin{align}\label{cr1}
    \Rchixm:=\frac{\dim(\lls_m(\DcNm))}{N!}\sum_{\gs\in \fS_N}\chix_m(\gs) \rho_\gs
  \end{align}
is the orthogonal projection of $\ell^2(\Nm))$ onto the subspace
$\lls_m(\DcNm)$. 
\end{corollary}

\begin{proof}
This is a special case of \refP{PA4}, with 
$V=\ell^2(\Nm)$ and $W=\lls_m(\DcNm)$;
since $W\subseteq V$ we may take $U_1=W$, and by \refT{TR2},
no other $U_i$ yields a representation $\rho\restr{U_i}$ 
isomorphic to $\rho\restr{W}$.
Note that $\chix_m$ is real (and integer-valued) as a consequence of 
\eqref{b5b} and \eqref{b2.5}, and thus we may ignore the complex conjugate
in \eqref{pa4}.
\end{proof}

We remark that \emph{any} character on the symmetric group $\fS_N$ is
real, and in fact integer-valued. (This does not hold for general groups.)

\subsection{Proof of \refT{T1}}\label{SSpf3}
We can now use the discrete results in \refS{Spf}
to prove the continuous analogue \refT{T1}.

Recall that we may identify $\ell^2(\Nm)$ with the subspace 
$\tell^2(\Nm)\subset L^2(\oi^m)$.
In this subsection we will freely use this identification; 
we identify stepfunctions constant on cubes $Q_{N,\bi}$ with functions on $\Nm$,
and we abandon the notation $\tell^2$.
The following lemma shows that then $\Phi_{m,k}$ defined in \eqref{Phi} and
its discrete version $\PhiN$ defined in \eqref{PhiN} ``almost'' agree; more
precisely, they agree on the set $\DcNm$
(or, regarded as functions in $L^2(\oi^m)$, on $\cDcNm$);
note that this set has measure $1-O(m^2/N)$. 

\begin{lemma}\label{LQK}
  Let $N\ge1$ and $0\le k\le m$.
If $f\in\lls(\Nk)\subseteq L^2(\oi^k)$, then $f\cdot\ettDcNk\in\lls(\DcNk)$ and
\begin{align}\label{qk1}
\Phi_{m,k}f\cdot\ettcDcNm
=
  \PhiN_{m,k}(f\cdot\ettDcNk)
\end{align}
as elements of $\lls(\Nm)\subseteq L^2(\oi^m)$.
\end{lemma}
\begin{proof}
  We have $f\cdot\ettDcNk\in\lls(\DcNk)$ by the definitions \eqref{lo0}--\eqref{lo1}.

By comparing \eqref{Phi} and \eqref{PhiN}, we see that for any
$\bi\in\DcNm$ and $x\in Q_{N,\bi}$ (defined in \eqref{QN}),
we have  
$\Phi_{m,k}f(x)=\PhiN_{m,k}(f\cdot\ettDcNk)(\bi)$.
Furthermore, by definition, 
$\PhiN_{m,k}(f\cdot\ettDcNk)(\bi)=0$ for $\bi\notin\DcNk$
(while $\Phi_{m,k}f(x)$ may be non-zero for $x\in Q_{N,\bi}$ in this case).
Hence, \eqref{qk1} holds.
\end{proof}

We continue with a lemma that contains most of the work in this subsection.

\begin{lemma}\label{LQL}
Let $m\ge0$ and let $f\in \Ls(\oi^m)\setminus \Lslex{m-1}(\oi^m)$.
Then, for every sufficiently large $N$ we have
\begin{align}\label{ql1}
  \lls_m(\DcNm)\subseteq\rhof.
\end{align}
\end{lemma}

\begin{proof}
The case $m=0$ is trivial (with $f$ constant), so we may assume $m\ge1$.
Fix $m\ge1$ and assume $N\ge2m$.
We let $C$ denote constants 
that may depend on $m$ but not on $N$;
these constants may be different at each occurrence.

By the definition of $\Ls_m(\oi^m)$ in \refSS{SS1}, or by \eqref{d3}, we have
$\Ls(\oi^m)=\Ls_m(\oi^m)\oplus\Lslex{m-1}(\oi^m)$. Hence,
$f=f_m+\fx$ for some $f_m\in\Ls_m(\oi^m)$ and $\fx\in\Lslex{m-1}(\oi^m)$.
By assumption, $f\notin\Lslex{m-1}(\oi^m)$, and thus $f_m\neq0$.

By \eqref{pn} and the symmetry of $f$,
we have
$P_Nf\in\lls(\Nm)$. 
Denote the operator $\Rchixm$ in \eqref{cr1} by $R_N$.
We have $P_Nf\in\rhof$ by \refL{L2}, and it follows by \eqref{cr1} that
also
\begin{align}\label{ql2}
  R_NP_N(f)\in\rhof.
\end{align}
We consider two cases.

\pfcase1{$R_N P_N f\neq0$.}
By \refC{CR1}, $R_NP_Nf\in\lls_m(\DcNm)$, and
by \refT{TR1}, the representation $\rho$ of $\fS_N$ on $\lls_m(\DcNm)$
is irreducible. 
Hence, 
\refP{PA1} shows that $\lls_m(\DcNm)=\rhox{R_NP_Nf}$.
Thus, the definition \eqref{hull} shows that linear combinations of
functions $\rho_g(R_N P_N f)$ with $g\in\fS_N$ are dense in $\lls_m(\DcNm)$.
It follows from \eqref{ql2} that every such linear combination is in $\rhof$,
and thus \eqref{ql1} holds.

\pfcase2{$R_N P_N f=0$.}
By \refC{CR1}, $R_N$ is the orthogonal projection onto $\lls_m(\DcNm)$.
Hence, we have $P_Nf\perp\lls_m(\DcNm)$.
It follows that for any $h\in\lls_m(\DcNm)$,
\begin{align}\label{qj2}
  0 = \innprod{P_N f,h} 
= \innprod{P_Nf,h\cdot\ettDcNm}
= \innprod{P_Nf\cdot \ettDcNm,h}
.\end{align}
Since $P_N f\cdot\ettDcNm\in \lls(\DcNm)$, this shows that, 
using \eqref{d2sN} and \eqref{ste12},
\begin{align}\label{qj3}
  P_Nf\cdot\ettDcNm\in\lls(\DcNm)\ominus\lls_m(\DcNm)
=\llslex{m-1}(\DcNm).
\end{align}
This means by the definition \eqref{a5} that there exists
$F_N\in\lls(\DcNx{m-1})$ such that 
\begin{align}\label{ql3}
  P_Nf\cdot\ettDcNm=\PhiN_{m,m-1}F_N.
\end{align}
Furthermore, by \refL{LR0},
we have
\begin{align}\label{ql4}
  \norm{F_N}_2 
\le C \norm{P_Nf\cdot\ettDcNm}_2
\le C\norm{P_Nf}_2
\le C\norm{f}_2
.\end{align}

By \eqref{ql3} and \refL{LQK}, we have, 
since $F_N\in\lls(\DcNx{m-1})$   and thus
$F_N=F_N\cdot\ettDcNmi$,
\begin{align}\label{ql5}
  P_Nf\cdot\ettDcNm
=\PhiN_{m,m-1}F_N
=\PhiN_{m,m-1}(F_N\cdot\ettDcNmi)
=\Phi_{m,m-1}F_N\cdot\ettcDcNm.
\end{align}
Furthermore, by \refL{LPhi},
$\Phi_{m,m-1}F_N\in \Ls_{m-1}(\oi^m)$, and thus $\Phi_{m,m-1}F_N\perp f_m$.
Consequently, using \eqref{ql5},
\begin{align}\label{ql6}
  0 &= \innprod{f_m,\Phi_{m,m-1}F_N}
= \innprod{f_m,\ettcDcNm\cdot\Phi_{m,m-1}F_N} 
+ \innprod{f_m,\ettcDNm\cdot\Phi_{m,m-1}F_N}
\notag\\&
= \innprod{f_m,\ettcDcNm\cdot P_Nf} 
+ \innprod{f_m,\ettcDNm\cdot\Phi_{m,m-1}F_N}
\notag\\&
= \innprod{f_m, P_Nf} 
- \innprod{\ettcDNm\cdot f_m, P_Nf} 
+ \innprod{\ettcDNm\cdot f_m,\Phi_{m,m-1}F_N}
.\end{align}
It is well-known that as \Ntoo, we have $P_N f\to f$ in $L^2(\oi^m)$.
Hence, 
\begin{align}\label{ql7}
\innprod{f_m, P_Nf}
\to
\innprod{f_m,f}
=\innprod{f_m,f_m}
=\norm{f_m}_2^2
>0.
\end{align}
Furthermore, $|\cDNm|\le \binom m2 N\qw\to0$ as \Ntoo, and thus
\begin{align}\label{ql8}
  \norm{\ettcDNm\cdot f_m}_2^2
=\int_{\cDNm}|f_m|^2 \to0.
\end{align}
The operator $P_N$ is an orthogonal projection and has norm 1, and
$\Phi_{m,m-1}$ in \eqref{Phi} is clearly bounded; thus
\eqref{ql4} implies
\begin{align}\label{qm1}
  \norm{-P_Nf + \Phi_{m,m-1}F_N}_2
\le \norm{f}_2+ C\norm{F_N}_2
\le C\norm{f}_2.
\end{align}
Consequently, we obtain by \eqref{ql8}, \eqref{qm1}, and the \CSineq, 
as \Ntoo,
\begin{align}\label{ql9}
 - \innprod{\ettcDNm\cdot f_m, P_Nf} 
+ \innprod{\ettcDNm\cdot f_m,\Phi_{m,m-1}F_N}
\to0.
\end{align}

We have shown in \eqref{ql7} and \eqref{ql9}
that as \Ntoo, the \rhs{} of \eqref{ql6} tends to 
$\norm{f_m}_2^2>0$.
This means that if $N$ is large enough, then \eqref{ql6} cannot hold.
This contradiction means that the assumption $R_NP_Nf=0$ then cannot hold;
in other words, if $N$ is large enough, then we must have $R_NP_Nf\neq0$
so Case 1 holds, and as shown above, then \eqref{ql1} holds.
\end{proof}

\begin{proof}[Proof of \refT{T1}]
We prove first that every closed invariant subspace
$M$ has a decomposition \eqref{t1}.
We use induction on $m$. The base case $m=0$ is trivial, so we may assume
$m\ge1$. 
Assume that $M$ is a closed $\rho$-invariant subspace of $\Ls(\oi^m)$.
We consider two cases.

\pfcase1{$M\subseteq\Lslex{m-1}(\oi^m)$}
By \refL{LPhi}, $\Phi_{m,m-1}$ is an isomorphism
$\Ls(\oi^{m-1})\to\Lslex{m-1}(\oi^m)$, and it follows from the definition
\eqref{Phi} that $\Phi_{m,m-1}$ intertwines the action $\rho$ of $\GG$ on
$L^2(\oi^{m-1})$ and $L^2(\oi^m)$.
Consequently, $\Phi_{m,m-1}\qw M$ is a closed $\rho$-invariant subspace of
$\Ls(\oi^{m-1})$, and by the induction hypothesis we have a decomposition
\begin{align}\label{t1q}
  \Phi_{m,m-1}\qw M 
=  \bigoplus_{k\in K} \Ls_k\bigpar{\oi^{m-1}}
\end{align}
for some set $K\subseteq\set{0,\dots,m-1}$.
We then obtain the desired decomposition \eqref{t1} by applying
$\Phi_{m,m-1}$ and using \refL{LPhi} again.

\pfcase2{$M\not\subseteq\Lslex{m-1}(\oi^m)$}
In this case there exists $f\in M\subseteq\Ls(\oi^m)$ such that
$f\notin\Lslex{m-1}(\oi^m)$. 
Since $M$ is invariant, we have $\rhof\subseteq M$, and thus \refL{LQL}
yields, for every sufficiently large $N$.
\begin{align}\label{qn1}
  \lls_m(\DcNm) \subseteq \rhof \subseteq M.
\end{align}

Consider now the orthogonal complement $M^\perp:=\Ls(\oi^m)\ominus M$.
This is also a closed $\rho$-invariant subspace of $\Ls(\oi^m)$, since $M$ is.
We have the same two cases for $M^\perp$.
If Case 1 holds, i.e., $M^\perp\subseteq \Lslex{m-1}(\oi^m)$,
then, as shown above, \eqref{t1} holds for $M^\perp$, i.e., 
\begin{align}\label{t1perp}
M^\perp=  \bigoplus_{k\in K^\perp} \Lsmk
\end{align}
for some $K^\perp\subseteq\set{0,1,\dots,m}$. 
This and \eqref{d3} (with $\ell=m$) imply that 
\eqref{t1} holds witk $K:=\set{0,\dots,m}\setminus K^\perp$.

Finally, suppose that Case 2 applies to both $M$ and $M^\perp$.
Then, for every large $N$, \eqref{qn1} holds (for some $f$), and also similarly
\begin{align}\label{qn3}
  \lls_m(\DcNm) \subseteq  M^\perp.
\end{align}
But this is absurd, since \eqref{qn1} and \eqref{qn3} imply that  
$\lls_m(\DcNm) \subseteq  M\cap M^\perp=\set0$.
Hence this case cannot occur, which completes the proof of \eqref{t1}.

Let $V_k:=\Ls_k(\oi^m)$, and 
suppose that $W$ is a closed invariant subspace of $V_k$
Then $W$ has a decomposition \eqref{t1}, and since $W\subseteq V_k$ it is
obvious that $K=\emptyset$ or $\set k$, and thus $W=\set0$ or $W=V_k$, which
shows that $\rho\restr V_k$ is irreducible.

Finally, suppose that two of these representations are
equivalent, say $\rho\restr V_i$ and $\rho\restr V_j$ with $0\le i<j\le m$.
Then there exists an invertible bounded linear  operator $T:V_i\to V_j$ such
that  
$T \rho_\gf = \rho_\gf T$ on $V_i$ for every $\gf\in\GG$.
This implies that $W:=\set{f+Tf:f\in V_i}$ is a closed invariant subspace of
$V_i\oplus V_j\subseteq \Ls(\oi^m)$. But $W\cap V_k=\set0$ for every $k$,
and thus $W$ cannot have a decomposition \eqref{t1}. This contradiction
shows that the representations are non-equivalent. 
\end{proof}

\appendix

\section{Group representations}\label{Arep}

In this appendix we collect for easy reference
some results on group representations that are used above. 
The results are known, 
and can be found in many textbooks, see for example
\cite{FultonHarris} and \cite{Serre}.

A \emph{representation} of a group $G$ on a vector space $V$
is a group homomorphism $\rho: g\mapsto\rho_g$
from $G$ to $\GL(V)$, the group of linear bijections (isomorphisms) $V\to V$.
In this paper, 
all representations are \emph{unitary}, 
which means that $V$ is a complex Hilbert space and
$\rho_g$ is a unitary operator in $V$ for every $g\in G$.
Furthermore, 
the vector spaces $V$ that we consider are 
always subspaces of $L^2(\cX,\mu)$ for some measure space $(\cX,\mu)$.
(Recall that $L^2(\cX,\mu)$ is the Hilbert space of square-integrable
complex-valued functions defined on $X$,
with functions that are a.e.\ equal identified.)
Note that we consider both finite-dimensional and infinite-dimensional Hilbert
spaces $V$.

Let $\rho$ be a unitary representation of $G$ on $V$.
We recall some notation from \refSS{SSnotF}.

A closed subspace $W\subseteq V$ is \emph{invariant} if $\rho_g(W)\subseteq W$
for every $g\in G$. (Then, in fact, $\rho_g(W)=W$.)
This means that the restriction of $\rho_g$ to $W$ defines a representation
$\rho\restr{W}$ of $G$ on $W$.

Given a representation $\rho$ of $G$ on $V$, and an element $f\in V$,
define
\begin{align}\label{hullA}
  \rhof := \text{the closed linear hull of }\set{\rho_g(f):g\in G}
\subseteq V.
\end{align}
It is easily seen that $\rhof$ is an invariant subspace of $V$; in fact, it
is the smallest invariant subspace that contains $f$.

A representation $\rho$  is \emph{irreducible} if the only closed invariant
subspaces are the trivial $\set{0}$ and $V$.
We record a well-known fact.
\begin{prop}\label{PA1}
  Let $\rho$ be a unitary representation of a group $G$ on a Hilbert space $V$.
Then the following are equivalent.
\begin{romenumerate}
  
\item \label{PA1a}
$\rho$ is irreducible.

\item \label{PA1b}
$\rhof=V$ for every $0\neq f\in V$.
\end{romenumerate}
\end{prop}
\begin{proof}
As said above,
$\rhof$ is a closed invariant subspace for every
  $f\in V$.
Hence, if $\rho$ is irreducible and $f\neq0$, then $\rhof=V$.

Conversely, if 
$W\neq\set0$ is a closed invariant subspace,
let $0\neq f\in W$. Since $W$ is invariant and closed, 
it follows that $\rhof\subseteq W$.
Consequently, \ref{PA1b} implies $W=V$ and thus $\rho$ is irreducible.
\end{proof}

\subsection{Finite-dimensional representations of finite groups}

In this subsection, we assume that $\rho$ is a finite-dimensional
representations of a finite group $G$. 
(In this case, there is always an inner product on $V$ that makes the
representation unitary, so this can be assumed without loss of generality.
In our applications the representations are already unitary for the given inner
products.) 

If $\rho$ is a representation of a finite group $G$ 
on a finite-dimensional space $V$, 
its \emph{character}
$\chi_\rho$ is defined as the function $G\to\bbC$ given by the trace 
of $\rho$:
\begin{align}\label{chi}
  \chi_\rho(g):=\Tr(\rho_g),
\qquad g\in G.
\end{align}
Two representations are isomorphic if and only if they have the
same character.

We regard the characters as elements of $L^2(G)$, where $G$ is equipped with
the normalized uniform (Haar) measure (i.e., $1/|G|$ times counting measure).
Hence, for two characters $\rho_1$ and $\rho_2$, their inner product is
\begin{align}\label{a00}
  \innprod{\chi_{\rho_1},\chi_{\rho_2}}
:=\frac{1}{|G|}\sum_{g\in G}\chi_{\rho_1}(g)\overline{\chi_{\rho_2}(g)}
.\end{align}
\begin{prop}[{\cite[Th{\'e}or{\`e}mes 3 and 5]{Serre}}]
\label{PA3}
A finite-dimensional representation of a finite group $G$
is irreducible if and only if its character has
norm $1$, and  different (i.e., non-isomorphic)
irreducible representations of $G$ have orthogonal characters.  
\end{prop}

\begin{prop}
[{\cite[Th{\'e}or{\`e}me 2]{Serre}; \cite[Corollary 1.6]{FultonHarris}}]
\label{PA5}
  Let $\rho$ be a representation of a finite group $G$ on a
  finite-dimen\-sional space $V$. 
Then there exists a decomposition $V=U_1\oplus\dotsm\oplus U_m$
of $V$ as a direct sum of invariant subspaces such
  that every $\rho\restr{U_i}$ is irreducible.
\end{prop}

\begin{prop}
[{\cite[Th{\'e}or{\`e}me 8]{Serre}; \cite[(2.31) p 23]{FultonHarris}}]
\label{PA4}
  Let $\rho$ be a representation of a finite group $G$ on a
  finite-dimen\-sional space $V$. Suppose that $V=U_1\oplus\dotsm\oplus U_m$
  is a decomposition of $V$ as a direct sum of invariant subspaces such
  that every $\rho\restr{U_i}$ is irreducible.
Let $\qrho$ be an irreducible finite-dimensional representation of $G$
on a vector space $W$ and let $\qchi$ be its character.
Then the linear operator
\begin{align}\label{pa4}
  R_{\qchi}:=\frac{\dim(W)}{|G|}\sum_{g\in G}\overline{\qchi(g)} \rho_g
  \end{align}
is the projection of $V$ onto the subspace
consisting of the sum of all $U_i$ such that $\rho\restr{U_i}\cong \qrho$.
If the representation $\rho$ is unitary (as we assume above), then the projection
$R_\qchi$ is orthogonal.
\end{prop}

\begin{example}\label{Echi}
  Let a group $G$ act on a finite set $S$ by permutations $\gs_g$, 
$g\in G$. This defines a representation $\rho$ of $G$ on $\ell^2(S)$ by
$\rho_g(f)(x):=f(\gs_g\qw(x))$ (cf.\ \eqref{rho}).
By considering the standard basis in $\ell^2(S)$, we see that the character
of this representation is given by
\begin{align}
  \chi_\rho(g) = |\set{x\in S: \gs_g(x)=x}|,
\end{align}
the number of fixed points of $\rho_g$.
\end{example}

\begin{example}\label{EYoung}
We give a simple proof that, as claimed in \refR{RYoung},
if $0\le k\le m$ and $N\ge m+k$, then
the  representation of $\fS_N$ 
given by $\rho\restr{\lls_k(\DcNm)}$ 
is the irreducible representation 
corresponding to the Young diagram (or the partition)
$(N-k,k)$. 
(This gives an alternative proof of \refT{TR1}.)

Let $\gs\in\fS_N$ and suppose that $\gs$ has $\njgs$ cycles of length $j$.
Thus $\sum_{j=1}^N j\njgs=N$.
A subset of $[N]$ is fixed by $\gs$ if and only if it is a union of cycles
of $\gs$; in other words, for every cycle we have the choice of either
including it or not. 
Hence, the generating function for the sizes of the fixed sets is
\begin{align}\label{k1}
  \sum_{A\subseteq[N]}\indic{\gs(A)=A} x^{|A|}
= \prod_{j=1}^N(1+x^j)^{\njgs}.
\end{align}
Let $[x^m]G(x)$ denote the coefficient of $x^m$ in a polynomial $G(x)$,
Then, by \eqref{b2.5} and \eqref{k1}, 
\begin{align}\label{k2}
  \chi_k(\gs) 
= [x^k]  \sum_{A\subseteq[N]}\indic{\gs(A)=A} x^{|A|}
= [x^k]\prod_{j=1}^N(1+x^j)^{\njgs},
\end{align}
and thus by \eqref{b5b}, 
\begin{align}\label{k3}
\chix_k(\gs)
=  \chi_k(\gs)-\chi_{k-1}(\gs) 
= [x^k]\biggpar{(1-x)\prod_{j=1}^N(1+x^j)^{\njgs}}.
\end{align}
We rewrite \eqref{k3} using homogeneous polynomials as
\begin{align}\label{k4}
\chix_k(\gs)
= [x_1^{N-k+1}x_2^k]\biggpar{(x_1-x_2)\prod_{j=1}^N(x_1^j+x_2^j)^{\njgs}},
\end{align}
which agrees with
the Frobenius character formula \cite[4.10 p.~49]{FultonHarris}
for the character defined by the Young diagram $(N-k,k)$
\end{example}

\begin{acks}
I thank Martin Herschend for enlightening discussions on group representations. 
\end{acks}

\newcommand\AAP{\emph{Adv. Appl. Probab.} }
\newcommand\JAP{\emph{J. Appl. Probab.} }
\newcommand\JAMS{\emph{J. \AMS} }
\newcommand\MAMS{\emph{Memoirs \AMS} }
\newcommand\PAMS{\emph{Proc. \AMS} }
\newcommand\TAMS{\emph{Trans. \AMS} }
\newcommand\AnnMS{\emph{Ann. Math. Statist.} }
\newcommand\AnnPr{\emph{Ann. Probab.} }
\newcommand\CPC{\emph{Combin. Probab. Comput.} }
\newcommand\JMAA{\emph{J. Math. Anal. Appl.} }
\newcommand\RSA{\emph{Random Structures Algorithms} }
\newcommand\DMTCS{\jour{Discr. Math. Theor. Comput. Sci.} }

\newcommand\AMS{Amer. Math. Soc.}
\newcommand\Springer{Springer-Verlag}
\newcommand\Wiley{Wiley}

\newcommand\vol{\textbf}
\newcommand\jour{\emph}
\newcommand\book{\emph}
\newcommand\inbook{\emph}
\def\no#1#2,{\unskip#2, no. #1,} 
\newcommand\toappear{\unskip, to appear}

\newcommand\arxiv[1]{\texttt{arXiv}:#1}
\newcommand\arXiv{\arxiv}

\newcommand\xand{and }
\renewcommand\xand{\& }

\def\nobibitem#1\par{}


\begin{thebibliography}{99}

\bibitem[Borgs, Chayes, \Lovasz, S\'os and Vesztergombi(2008)]{BCLSV1}
Christian Borgs, Jennifer T. Chayes, L{\'a}szl{\'o} \Lovasz, Vera T. S\'os
\& Katalin Vesztergombi. 
Convergent sequences of dense graphs I: Subgraph
frequencies, metric properties and testing,
\emph{Advances in Math.} {\bf 219} (2008), 1801--1851.

\bibitem[Borgs, Chayes, \Lovasz, S\'os and Vesztergombi(2012)]{BCLSV2}
Christian Borgs, Jennifer T. Chayes, L{\'a}szl{\'o} \Lovasz, Vera T. S\'os
\& Katalin Vesztergombi. 
Convergent sequences of dense graphs II. Multiway cuts and statistical
physics. 
\emph{Ann. of Math. (2)} \vol{176} (2012), no. 1, 151--219.

\bibitem[Chung and Graham(1992)]{ChungG}
Fan R. K. Chung \& Ronald L. Graham.
Maximum cuts and quasirandom graphs.  
\emph{Random graphs, Vol. 2 (Pozna\'n, 1989)},  23--33, 
Wiley, New York, 1992. 

\bibitem[Chung, Graham and Wilson(1989)]{ChungGW:quasi}
Fan R. K. Chung, Ronald L. Graham \&  Richard M. Wilson. 
Quasi-random graphs.  
\jour{Combinatorica}  \vol9  (1989),  no. 4, 345--362. 

\bibitem{FultonHarris}
William Fulton \& Joe Harris.
\emph{Representation Theory. A First Course.} 
Springer-Verlag, New York, 1991. 

\bibitem{CM}
Ronald L. Graham, Donald E. Knuth \& Oren Patashnik.
\emph{Concrete Mathematics}.
2nd ed, Addison-Wesley,  Reading, MA, 1994.

\bibitem[Hatami, Hatami and Li(2014)]{Hatami2Li}
Hamed  Hatami, Pooya Hatami \& Yaqiao  Li. 
A characterization of functions with vanishing averages over products of
disjoint sets. 
\emph{European J. Combin.} \vol{56} (2016), 81--93. 

\bibitem[Huang and Lee(2012)]{HuangLee}
Hao Huang \& Choongbum Lee.
Quasi-randomness of graph balanced cut properties.
\emph{Random Structures  Algorithms}
\vol{41} (2012), no. 1,  124--145.

\bibitem[Janson(2011)]{SJ234}
Svante Janson.
Quasi-random graphs and graph limits.
\emph{European J. Combin.} \vol{32} (2011), 1054--1083.

\bibitem[Janson and S{\'o}s(2015)]{SJ294}
Svante Janson \& Vera T. S{\'o}s.
More on quasi-random graphs, subgraph counts and graph limits.
\emph{European J. Combin.} \textbf{46} (2015), 134--160.

\bibitem[\Lovasz{}(2012)]{Lovasz}
L{\'a}szl{\'o} Lov{\'a}sz.
\emph{Large Networks and Graph Limits}.
American Mathematical Society, Providence, RI, 2012.

\bibitem[\Lovasz{} and Szegedy(2006)]{LSz}
L{\'a}szl{\'o} Lov{\'a}sz \& Bal{\'a}zs Szegedy.
Limits of dense graph sequences. 
\emph{J. Combin. Theory B} \vol{96} (2006), no. 6, 933--957.

\bibitem{MazurO}
 S. Mazur and W. Orlicz. 
Grundlegende Eigenschaften der polynomischen Operationen. 
\emph{Studia Math.} \vol5 (1934), 50--68.

\bibitem[Petrov(2013)]{Petrov}
Feodor Petrov.
General removal lemma.
Preprint, 2013.
\arxiv{1309.3795v1}

\bibitem{Serre}
Jean-Pierre Serre. 
\emph{Représentations linéaires des groupes finis}. 
2nd ed., 
Hermann, Paris, 1971. 

\bibitem[Shapira(2008)]{Shapira}
Asaf Shapira.
Quasi-randomness and the distribution of copies of a fixed graph. 
\emph{Combinatorica} \vol{28} (2008), 735--745. 


\bibitem[Shapira and Yuster(2010)]{ShapiraYuster}
Asaf Shapira \& Raphael Yuster.
The effect of induced subgraphs on quasi-randomness. 
\emph{Random Struct Algorithms} \vol{36} (2010), 90--109.


\bibitem[Shapira and Yuster(2012)]{ShapiraYuster:hyper}
Asaf Shapira \& Raphael Yuster.
The quasi-randomness of hypergraph cut properties. 
\emph{Random Structures Algorithms} \vol{40} (2012), no. 1, 105--131.

\bibitem[Simonovits and S\'os(1997)]{SS:nni}
Mikl{\'o}s Simonovits \& Vera T. S\'os.
Hereditarily extended properties, quasi-random graphs and not necessarily
induced subgraphs. 
\emph{Combinatorica} \vol{17} (1997), no. 4, 577--596. 

\bibitem[Simonovits and S\'os(2003)]{SS:ind}
Mikl{\'o}s Simonovits \& Vera T. S\'os.
Hereditary extended properties, quasi-random graphs and induced
subgraphs. 
\emph{Combin. Probab. Comput.}  \vol{12}  (2003),  no. 3, 319--344.

\bibitem{StanleyI}
Richard P. Stanley.
\emph{Enumerative Combinatorics, Volume I}.
Cambridge Univ. Press, Cambridge, 1997.

\bibitem{Thomas}
Erik G. F. Thomas.
A polarization identity for multilinear maps.
With an appendix by Tom H. Koornwinder.
\emph{Indag. Math. (N.S.)} \vol{25} (2014), no. 3, 468--474.

\bibitem[Thomason(1987)]{Thomason87a}
Andrew Thomason.
Pseudorandom graphs. 
\inbook{Random graphs '85 (Pozna\'n, 1985)}, 307--331, 
North-Holland, Amsterdam, 1987.

\bibitem[Thomason(1987)]{Thomason87b}
Andrew Thomason.  
Random graphs, strongly regular graphs and pseudorandom graphs. 
\inbook{Surveys in Combinatorics 1987 (New Cross, 1987)}, 173--195, 
London Math. Soc. Lecture Note Ser. \vol{123}, 
Cambridge Univ. Press, Cambridge, 1987. 

\bibitem[Yuster(2008)]{Yuster}
Raphael Yuster.
Quasi-randomness is determined by the distribution of copies of a
fixed graph in equicardinal large sets.
\inbook{Approximation, Randomization and Combinatorial Optimization},
596--601,
Lecture Notes in Comput. Sci. \vol{5171}, Springer, Berlin, 2008. 



\end{thebibliography}
\end{document}